\documentclass{amsart}
\usepackage[alphabetic]{amsrefs}
\usepackage{amssymb}
\usepackage{subfig}
\usepackage{wrapfig}
\usepackage[pdftex]{graphicx}
\usepackage{color}
\usepackage{amsfonts}
\usepackage{sidecap}
\usepackage{eucal}

\newcommand{\A}{\mathcal{A}}
\newcommand{\C}{\mathcal{C}}
\newcommand{\R}{\mathbb{R}}
\newcommand{\N}{\mathbb{N}}
\newcommand{\cN}{\mathcal{N}}

\newcommand{\supp}{\operatorname{supp}}

\newcommand{\capacity}{\operatorname{cap}}
\renewcommand{\Re}{\operatorname{Re}}

\newcommand{\dist}{\operatorname{dist}}
\newcommand{\case}[1]{\par\medskip\par\noindent\emph{#1}} 
\newcommand{\step}[1]{\par\medskip\par\noindent\emph{#1}} 

\newcommand{\avint}{-\kern-10.7pt\int}

\def\XXint#1#2#3{{\setbox0=\hbox{$#1{#2#3}{\int}$}
\vcenter{\hbox{$#2#3$}}\kern-0.5\wd0}}

\renewcommand{\phi}{\varphi}
\renewcommand{\epsilon}{\varepsilon}

\theoremstyle{plain}
\newtheorem{theorem}{Theorem}[section]
\newtheorem{lemma}[theorem]{Lemma}
\newtheorem{proposition}[theorem]{Proposition}
\newtheorem{corollary}[theorem]{Corollary}
\theoremstyle{definition}

\theoremstyle{remark}
\newtheorem{remark}[theorem]{Remark}
\numberwithin{equation}{section}

\allowdisplaybreaks[1]

\title[{Parabolic Domains with Thin
  Lipschitz Complement}]{Parabolic Boundary Harnack Principles in Domains with Thin
  Lipschitz Complement}
\author{Arshak Petrosyan}
\address{Department of Mathematics, Purdue University, West Lafayette,
  IN 47907}
\email{arshak@math.purdue.edu}
\thanks{The authors were supported in part by NSF grant DMS-1101139}
\author{Wehnui Shi}
\address{Mathematisches Institut, Universit\"{a}t Bonn, Endenicher Allee 64,  53115 Bonn,
Germany}
\email{wenhui.shi@hcm.uni-bonn.de}

\begin{document}
\begin{abstract}
 We prove forward and backward parabolic boundary Harnack principles
 for nonnegative solutions of the heat equation in the complements of
 thin parabolic Lipschitz sets given as subgraphs
$$
E=\{(x,t): x_{n-1}\leq f(x'',t),x_n=0\}\subset \R^{n-1}\times\R
$$
for parabolically Lipschitz functions $f$ on $\R^{n-2}\times\R$.

We are motivated by applications to parabolic free boundary
problems with thin (i.e co-dimension two) free boundaries. In particular, at the end of the
paper we show how to prove the spatial $C^{1,\alpha}$ regularity of the
free boundary in the parabolic Signorini problem.
\end{abstract}

\maketitle
\section{Introduction}
\label{sec:introduction}
The purpose of this paper is to study forward and backward boundary
Harnack principles for nonnegative solutions of the heat equation
in a certain type of domains in $\R^{n}\times\R$, which are, roughly
speaking, complements of thin parabolically Lipschitz sets $E$. By
the latter we understand closed sets, lying in the vertical hyperplane $\{x_n=0\}$, and which are locally given as subgraphs of parabolically
Lipschitz functions
 (see Fig.~\ref{fig:parabolic}). 

This kind of sets appear naturally in free boundary problems governed
by parabolic equations, where the free boundary lies in a given
hypersurface and thus has co-dimension two. Such free boundaries are
also known as thin free boundaries.  In particular, our study was
motivated by the
parabolic Signorini problem, recently studied in \cite{DGPT}.

The boundary Harnack
principles that we prove in this paper provide important technical
tools in problems with thin free boundaries. For instance, they open
up the possibility for proving that the thin Lipschitz free boundaries
have H\"older continuous spatial normals, following the original idea in
\cite{AC}. In particular, we show that this argument indeed can be
successfully carried out in the parabolic Signorini problem.

We have to point out that the elliptic counterparts of the results in
this paper are very well known, see e.g.\ \cite{AC,CSS,ALM}. However, 
there are significant differences between the elliptic and parabolic
boundary Harnack principles, mostly because of the time-lag in the
parabolic Harnack inequality. This results in two types of
the boundary Harnack principles for the parabolic equations: the forward
one (also known as the Carleson estimate) and the 
backward one. Besides, those results are known only for a much smaller
class of domains than in the elliptic case.  Thus, to put our results
in a better perspective, we start with a discussion of the known
results both in the elliptic and parabolic cases.

\subsection*{Elliptic boundary Harnack principle}

By now classical boundary Harnack principle for harmonic functions \cite{Ke-ell,Dahl,Wu-ell} says
that if $D$ is a bounded Lipschitz domain in $\R^n$, $x_0\in \partial
D$, and  $u$ and $v$ are positive harmonic functions on $D$ vanishing on
$B_{r}(x_0)\cap \partial D$ for a small $r>0$, then there
exist positive constants $M$ and $C$, depending only on the dimension
$n$ and the Lipschitz constant of $D$, such that
\begin{equation*}
\frac{u(x)}{v(x)}\leq C\frac{u(y)}{v(y)} \quad \text{for } x,y\in B_{r/M}(x_0)\cap D.
\end{equation*}
 Note that this result is scale-invariant, hence by a standard
 iterative argument, one then immediately obtains that the ratio $u/v$
 extends to $\overline{D}\cap B_{r/M}(x_0)$ as a H\"older continuous function. Roughly speaking, this theorem says that two positive harmonic
functions vanishing continuously on a certain part of the boundary
will decay at the same rate near that part of the boundary. 

The above boundary Harnack principle depends heavily on the geometric
structure of the domains. The scale
invariant boundary Harnack principle (among other classical theorems
of real analysis) was extended by \cite{JK} from
Lipschitz domains to the so-called NTA (non-tangentially accessible)
domains. Moreover, if the Euclidean 
metric is replaced by the internal metric, then similar results
hold for so-called uniform John domains \cite{ALM,HA}.

In particular, the boundary Harnack principle is known for the domains
of the following type
\begin{equation*}
D=B_1\setminus E_f, \quad E_f=\{x\in \R^n :  x_{n-1}\leq f(x''), x_{n} = 0\},
\end{equation*}
where $f$ is a Lipschitz function on $\R^{n-2}$, with $f(0)=0$, where
it is used for instance in the thin obstacle problem
\cites{AC,ACS2,CSS}. In fact, there is a relatively simple proof of
the boundary Harnack principle for the domains as above, already
indicated in \cite{AC}: there exists
a bi-Lipschitz transformation from $D$ to a halfball $B_1^+$, which is
a Lipschitz domain. The harmonic functions in $D$ transform to solutions of
a uniformly elliptic equation in divergence form with bounded measurable
coefficients in $B_1^+$, for which the boundary Harnack principle is
known \cite{CFMS}.

\subsection*{Parabolic boundary Harnack principle} The parabolic
version of the boundary Harnack principle is much more challenging
than the elliptic one, mainly because of the time-lag issue in the parabolic
Harnack inequality. The latter is called sometimes the forward Harnack
inequality, to emphasize the way it works: for nonnegative caloric
functions (solutions of the heat equation), if the earlier value is positive at some spatial point, after a necessary
waiting time, one can expect that the value will become positive
everywhere in a compact set containing that point. Under the condition
that the caloric function vanishes on the lateral boundary of the
domain, one may overcome the time-lag issue and get a backward type
Harnack principle (so combining together one gets an elliptic-type
Harnack inequality) 

The forward and backward boundary Harnack principle are known for
parabolic Lipschitz domains, not necessarily
cylindrical, see \cite{Kemper,Garofalo3,Salsa}. Moreover, they were
shown more recently in \cite{HLN} to hold for unbounded parabolically Reifenberg flat domains. In this paper, we will generalize 
the parabolic boundary Harnack principle to the domains of the
following type (see Figure \ref{fig:parabolic}): 
$$
D=\Psi_1\setminus E_f,
$$
where 
\begin{align*}\Psi_1&=\{(x,t):|x_i|<1,i=1,\ldots,n-2, |x_{n-1}|<4nL, |x_n|<1, |t|<1\};\\
E_f&=\{(x,t):x_{n-1}\leq f(x'',t), x_n=0\}
\end{align*}
and $f(x'',t)$ is a parabolically Lipschitz function satisfying
$$
|f(x'',t)-f(y'',s)|\leq L(|x''-y''|^2+|t-s|)^{1/2}; \quad f(0,0)=0.
$$
\definecolor{lightblue}{rgb}{0.863,0.816,0.867}%
\definecolor{lightpink}{rgb}{0.824,0.672,0.66}%
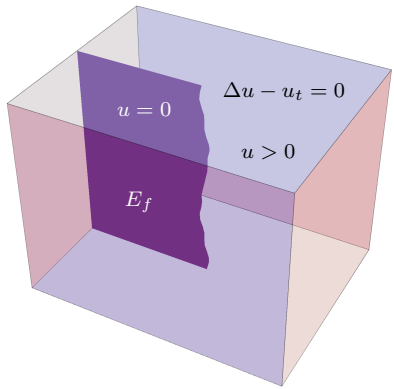
\begin{figure}
\begin{picture}(152,150)
\put(0,0){\includegraphics[height=150pt]{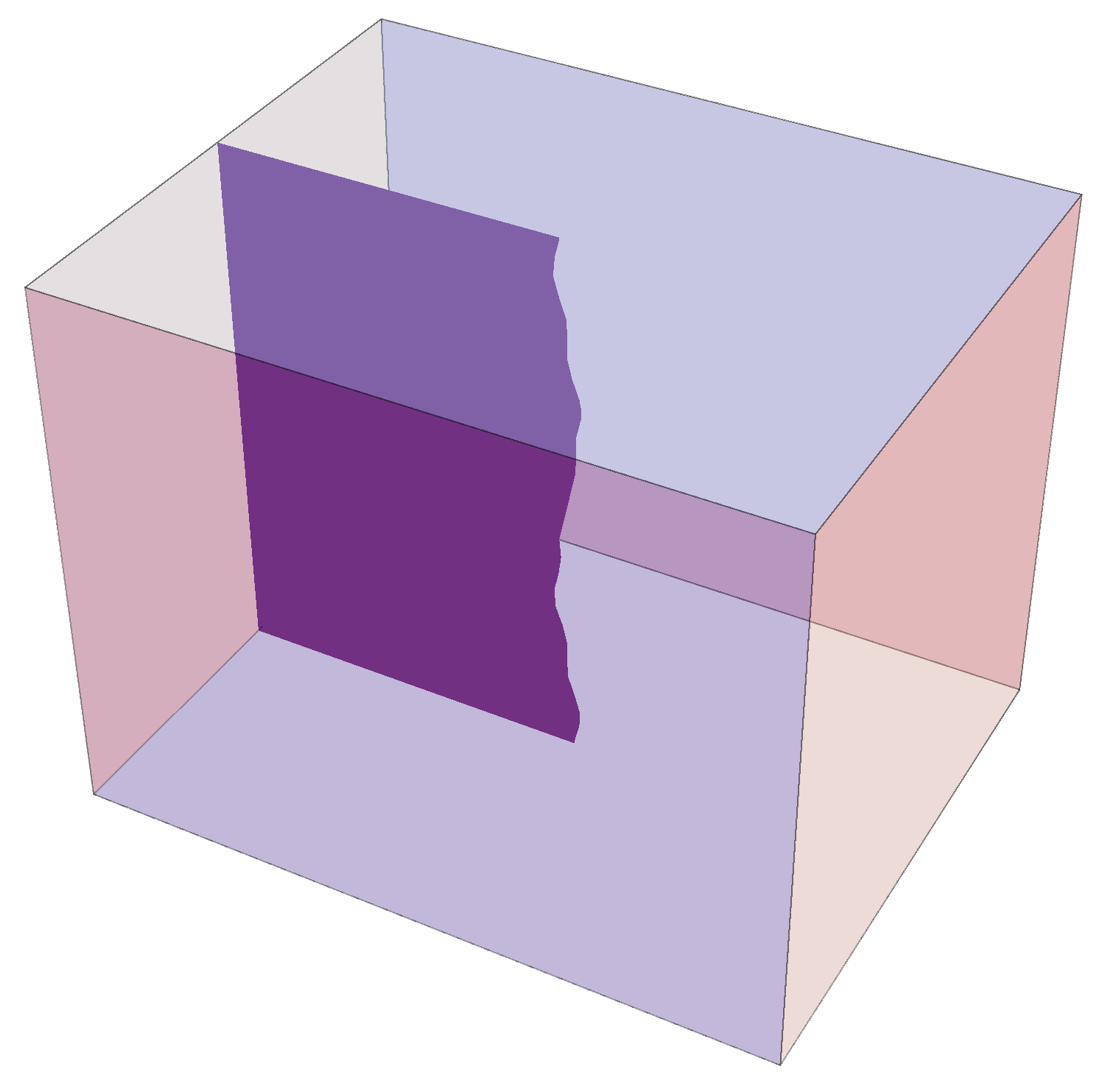}}
\put(48,72){\footnotesize \color{white}{$E_f$}}
\put(45,106){\footnotesize \color{white}{$u=0$}}
\put(92,90){\footnotesize $u>0$}
\put(85,113){\footnotesize $\Delta u-u_t=0$}
\end{picture}
\caption{Domain with a thin Lipschitz complement}
\label{fig:parabolic}
\end{figure}
Note that $D$ is not cylindrical ($E_f$ is not time invariant), and it
does not fall into any category of the domains on which the forward or
backward Harnack principle is known. Inspired by the elliptic inner
NTA domains (see e.g. \cite{ACS2}), it seems natural to equip the domain $D$ with the
intrinsic geodesic distance $\rho _D((x,t),(y,s))$, where $\rho
_D((x,t),(y,s))$ is defined as the infimum of the Euclidean length of rectifiable curves $\gamma$ joining $(x,t)$ and $(y,s)$ in $D$, and consider
the abstract completion $D^*$ of $D$ with respect to this inner
metric $\rho_D$. We will not be working directly with the inner
metric in this paper,  since it seems easier to work with the
Euclidean parabolic cylinders due to the time-lag issues and different
scales in space and time variables. However, we do use the fact that
the interior points of $E_f$ (in relative topology) correspond to two
different boundary points in the completion $D^*$.

Even though we assume in this paper that  $E_f$ lies on the hyperplane $\{x_{n}=0\}$ in
$\R^n\times\R$, our proofs, except those on the doubling of the caloric
measure and
the backward boundary Harnack principle, are easily generalized to the case
when $E_f$ is a hypersurface which is Lipschitz in space variable and
independent of time variable. 

\subsection*{Structure of the paper}
The paper is organized as follows.

In Section~\ref{sec:notat-prel} we give basic definitions and
introduce the notations used in this paper.

In Section~\ref{sec:barriers} we consider the Perron-Wiener-Brelot
(PWB) solution to the Dirichlet problem of the heat equation for
$D$. We show that $D$ is regular and has a H\"older continuous barrier
function at each parabolic boundary point. 

In Section~\ref{sec:forw-bound-harn} we establish a forward boundary Harnack
inequality for nonnegative caloric functions vanishing continuously on a
part of the lateral boundary following the lines of Kemper's paper
(\cite{Kemper}).  

In Section~\ref{sec:kernel-functions} we study the kernel functions
for the heat operator. We show that each boundary point $(y,s)$ in the
interior of $E_f$ (as a subset of the hyperplane $\{x_{n}=0\}$)
corresponds to two independent kernel functions. Hence the parabolic
Euclidean boundary for $D$ is not homeomorphic to the parabolic Martin
boundary.  

In Section~\ref{sec:backw-bound-harn}
we show the doubling property of the caloric measure with respect to $D$, which will imply a backward Harnack inequality for caloric functions vanishing on the whole lateral boundary.

Section~\ref{sec:applications} is dedicated to various forms of the
boundary Harnack principle from Sections~\ref{sec:forw-bound-harn} and
\ref{sec:backw-bound-harn}, including a version for solutions of the
heat equation with a nonzero right-hand side. We conclude the section
and the paper with an application to the parabolic Signorini problem.

\section{Notation and Preliminaries}
\label{sec:notat-prel}

\subsection{Basic Notation}
\begin{align*}
&\R^n && \text{the $n$-dimensional Euclidean space}\\
&x'=(x_1,\ldots,x_{n-1})\in\R^{n-1} && \text{for $x=(x_1,\ldots,
  x_n)\in\R^n$}\\
&x''=(x_1,\ldots,x_{n-2})\in\R^{n-2} &&\text{for $x=(x_1,\ldots,
  x_n)\in\R^n$}\\
\intertext{Sometimes it will be convenient to identify $x'$, $x''$ with
  $(x',0)$ and $(x'', 0,0)$, respectively.}
& x\cdot y=\sum_{i=1}^n x_iy_i,&&\text{the inner product for
  $x,y\in\R^n$}\\
&|x|=(x\cdot x)^{1/2}&&\text{the Euclidean norm of $x\in\R^n$}\\
&\|(x,t)\|=(|x|^2+|t|)^{1/2}&&\text{the parabolic norm of $(x,t)\in\R^n\times\R$}\\
&\overline E, E^\circ,\partial E&&\text{the closure, the
interior, the boundary of $E$}\\ 
&\partial_p E && \text{the parabolic boundary of $E$ in
  $\R^n\times\R$}\\
& B_r(x):=\{y\in\R^n: |x-y|<r\}&&\text{open ball in $\R^n$}\\
& B'_r(x'), B''_r(x'')&&\text{(thin) open balls in $\R^{n-1}$,
  $\R^{n-2}$}\\
&
\begin{aligned}
& Q_r(x,t):=B_r(x)\times (t-r^2,t)\\
\end{aligned}
&& 
\begin{aligned}\text{lower parabolic cylinders in }&\R^{n}\times\R
\end{aligned}\\
&\dist_p(E, F)=\inf_{\substack{(x,t)\in E\\(y,s)\in F}} \|(x-y,t-s)\| &&\text{the
  parabolic distance between sets $E$, $F$}
\end{align*}
We will also need the notion of \emph{parabolic Harnack chain} in a domain
$D\subset\R^n\times\R$. For two points $(z_1,h_1)$ and $(z_2,h_2)$ in
$D$ with $h_2-h_1\geq \mu^2 |z_2-z_1|^2$,  $0<\mu<1$,
we say that a sequence of parabolic cylinders $Q_{r_i}(x_i,t_i)\subset
D$,
$i=1,\ldots,N$ is a Harnack 
chain from $(z_1,h_1)$ to $(z_2,h_2)$ with a constant $\mu$ if

\begin{align*}
&(z_1,h_1)\in Q_{r_1}(x_1,t_1),\quad (z_2,h_2)\in Q_{r_N}(x_N,t_N)\\
&\mu\,
 r_i \leq \dist_p(Q_{r_i}(x_i,t_i),\partial_pD)\leq \frac1\mu
r_i,\quad i=1,\ldots,N,\\
&Q_{r_{i+1}}(x_{i+1},t_{i+1})\cap Q_{r_i}(x_i,t_i)\not=\emptyset,\quad i=1,\ldots,N-1,\\
&t_{i+1}-t_i\geq \mu^2 r_{i}^2,\quad i=1,\ldots,N-1.
\end{align*}
The number $N$ is called the length of the Harnack chain. By the
parabolic Harnack inequality, if $u$ is a nonnegative caloric function
in $D$ and there is a Harnack chain of length $N$ and constant $\mu$
from $(z_1,h_1)$ to $(z_2,h_2)$, then
$$
u(z_1,h_1)\leq C(\mu, n, N)\, u(z_2,h_2).
$$ 
Further, for given $L\geq 1$ and $r>0$ we also introduce the (elongated)
parabolic boxes, specifically adjusted to our purposes 
\begin{align*}
\Psi_r''&=\{(x'',t)\in\R^{n-2}\times\R:|x_i|< r, i=1,\ldots, n-2, |t|< r^2\} \\
\Psi'_r&=\{(x',t)\in\R^{n-1}\times\R: (x'',t)\in\Psi_r'', 
 |x_{n-1}|<4nLr\}\\
\Psi_r&=\{(x,t)\in\R^{n}\times\R:(x',t)\in\Psi_r', 
 |x_{n}|<r\} \\
\Psi_r(y,s)&=(y,s)+\Psi_r.
\end{align*}
We also define the following neighborhoods
\begin{alignat*}{2}
&\cN_r(E):=\bigcup_{(y,s)\in E} \Psi_r(y,s),&\quad&\text{for any set $E\subset
  \R^n\times \R$.}
\end{alignat*}

\subsection{Domains with thin Lipschitz complement}
\label{sec:domains-with-thin}

Let $f:\R^{n-2}\times\R\to \R$ be a parabolically Lipschitz function
with a Lipschitz constant $L\geq 1$ in a sense that
\[
|f(x'',t)-f(y'',s)|\leq L (|x''-y''|^2+|t-s|)^{1/2},\quad (x'',t),
(y'',s)\in\R^{n-2}\times\R
\]
Then consider the following two sets:
\begin{align*}
G_f&=\{(x,t):  x_{n-1}=f(x'',t),x_n=0\}\\
E_f&=\{(x,t): x_{n-1}\leq f(x'',t),x_n=0\}
\end{align*}
We will call them \emph{thin Lipschitz graph} and \emph{subgraph}
respectively (with ``thin'' indicating their lower dimension). We are
interested in a behavior of caloric functions in domains of the type
$\Omega\setminus E_f$, where $\Omega$ is open in $\R^n\times\R$. We
will say that $\Omega\setminus E_f$ is a domain
with a \emph{thin Lipschitz complement}.

We are interested mostly in local behavior of caloric functions near
the points on $G_f$ and therefore we concentrate our study on the
case
$$
D=D_f:=\Psi_1\setminus E_f
$$
with a normalization condition
$$
f(0,0)=0\iff (0,0)\in G_f.
$$

We will state most of our results for $D$ defined as above, however,
the results will still hold, if we replace $\Psi_1$ in the
construction above with a rectangular box 
$$
\tilde \Psi=\Big(\prod_{i=1}^n(a_i,b_i)\Big)\times(\alpha,\beta)
$$
such that for some
constants $c_0, C_0>0$ depending on $L$ and $n$, we have
$$
\tilde \Psi\subset \Psi_{C_0},\quad
\Psi_{c_0}(y,s)\subset\tilde\Psi,\quad \text{for all }(y,s)\in G_f,\ s\in[\alpha+c_0^2,\beta-c_0^2] 
$$
and consider the complement 
$$
\tilde D=\tilde D_f:=\tilde \Psi\setminus E_f.
$$
Even more generally, one may take $\tilde\Psi$ to be a
cylindrical domain of the type $\tilde \Psi=\mathcal{O}\times(\alpha,\beta)$ where
$\mathcal{O}\subset\R^n$ has the property that $\mathcal{O}_\pm=\mathcal{O}\cap\{\pm x_{n}>0\}$
are Lipschitz domains. For instance, we can take $\mathcal{O}=B_1$. Again, most of the results that we state will
be valid also in this case, with a possible change in constants that
appear in estimates.

\subsection{Corkscrew points}

Since will be working in $D=\Psi_1\setminus E_f$ as above, it will be
convenient to redefine sets $E_f$ and $G_f$ as follows:
\begin{align*}
G_f&=\{(x,t)\in\overline \Psi_1: x_{n-1}=f(x'',t),x_n=0\},\\
E_f&=\{(x,t)\in\overline \Psi_1: x_{n-1}\leq f(x'',t),x_n=0\},
\end{align*}
so that they are subsets of $\overline \Psi_1$. It is easy to see from
the definition of $D$ that it is connected and its parabolic boundary
is given by
$$
\partial_p D=\partial_p\Psi_1\cup E_f.
$$
As we will see, the domain $D$ has a parabolic NTA-like structure, with the
catch that at points on $E_f$ (and close to it) we need to define two
pairs of future and past corkscrew points, pointing into $D_+$ and
$D_-$ respectively, where
$$
D_+=D\cap\{x_{n}>0\}=(\Psi_1)_+,\quad D_-=D\cap\{x_{n}<0\}=(\Psi_1)_-.
$$
More specifically, fix $0<r<1/4$ and $(y,s)\in
\cN_r(E_f)\cap \partial_p D$, define
\begin{alignat*}{2}
\overline{A}^\pm_r(y,s)&=(y'', y_{n-1}+ r/2, \pm r/2, s+2r^2),&\quad&\text{if
}s\in [-1, 1-4r^2),\\
\underline{A}^\pm_r(y,s)&=(y'', y_{n-1}+r/2, \pm r/2, s-2r^2),&&\text{if
}s\in (-1+4r^2, 1].
\end{alignat*}
Note that by definition, we always have $\overline
A^+_r(y,s)$, $\underline A^+_r(y,s)\in D_+$ and  $\overline
A^-_r(y,s)$, $\underline A^-_r(y,s)\in D_-$. We also have that 
\begin{align*}
&\overline A^\pm_r(y,s), \underline A^\pm_r(y,s) \in \Psi_{2r}(y,s),\\
&\Psi_{r/2}(\overline A^\pm_r(y,s)), \Psi_{r/2}(\underline
A^\pm_r(y,s))\cap \partial D=\emptyset.
\end{align*}
Moreover, the corkscrew points have the following property.
\begin{lemma}[Harnack chain property I]\label{lem:H-chain-prop-I} Let
  $0<r<1/4$, $(y,s)\in \partial_p D\cap \cN_r(E_f)$, 
  and  $(x,t)\in D$ be such that
$$
(x,t)\in \Psi_r(y,s)\quad\text{and}\quad \Psi_{\gamma r}(x,t)\cap \partial_p D=\emptyset.
$$
Then there exists a Harnack chain in $D$ with a constant
$\mu$ and length $N$, depending only on $\gamma$, $L$, and $n$, from
$(x,t)$ to either $\overline A^+_r(y,s)$ or $\overline A^-_r(y,s)$,
provided $s\leq 1-4r^2$, and
from either $\underline A^+_r(y,s)$ or $\underline A^-_r(y,s)$ to
$(x,t)$, provided $s\geq -1+4r^2$. 

In particular, there exists a constant $C=C(\gamma,L,n)>0$ such that for any
nonnegative caloric function $u$ in $D$
\begin{alignat*}{2}
u(x,t)
&\leq
 C\,\max\{u(\overline A^+_r(y,s)),u(\overline A^-_r(y,s))\},&&\quad \text{if }s\leq 1-4r^2,\\
 u(x,t)&\geq C^{-1}\, \min\{u(\underline A^+_r(y,s)),(\underline A^-_r(y,s))\},&&\quad \text{if }s\geq -1+4r^2.
\end{alignat*}
\end{lemma}

\begin{proof}
This is easily seen when $(y,s)\not \in \cN_r(G_f)$ (in this case the chain length $N$ does not
depend on $L$). When $(y,s)\in \cN_r(G_f)$, one needs to use the
parabolic Lipschitz continuity of $f$.
\end{proof}

Next, we want to define the corkscrew points when $(y,s)$ is further
away for $E_f$. Namely, if $(y,s)\in \partial_p D\setminus
(\mathcal{N}_r(E_f))$, we define a single pair of future and past
corkscrew points by
\begin{alignat*}{2}
\overline {A}_r(y,s)&=(y(1-r),s+2r^2),&\quad& \text{if }s\in [-1,1-4r^2)\\
\underline {A}_r(y,s)&=(y(1-r),s-2r^2),&& \text{if } s\in
(-1+4r^2,1].
\end{alignat*}
Note that the points $\overline A_r(y,s)$ and  $\underline A_r(y,s)$
will have properties similar to those of $\overline A_r^\pm(y,s)$ and
$\underline A_r^\pm(y,s)$. That is, 
\begin{align*}
&\overline A_r(y,s), \underline A_r(y,s) \in \Psi_{2r}(y,s),\\
&\Psi_{r/2}(\overline A_r(y,s)), \Psi_{r/2}(\underline
A_r(y,s))\cap \partial D=\emptyset,
\end{align*}
and we have the following version of Lemma~\ref{lem:H-chain-prop-I} above
\begin{lemma}[Harnack chain property II]\pushQED{\qed}
\label{lem:H-chain-prop-II}
 Let $r\in(0,1/4)$, $(y,s)\in \partial_p D\setminus \cN_r(E_f)$
  and  $(x,t)\in D$ be such that
$$
(x,t)\in \Psi_r(y,s)\quad\text{and}\quad \Psi_{\gamma r}(x,t)\cap \partial_p D=\emptyset.
$$
Then there exists a Harnack chain in $D$ with a constant
$\mu$ and length $N$, depending only on $\gamma$, $L$, and $n$, from
$(x,t)$ to $\overline A_r(y,s)$, provided $s\leq 1-4r^2$, 
and
from $\underline A_r(y,s)$ to $(x,t)$, provided $s\geq -1+4r^2$.

In particular, there exists a constant $C=C(\gamma,L,n)>0$ such that for any
nonnegative caloric function $u$ in $D$
\begin{alignat*}{2}
&u(x,t)\leq C\,u(\overline A_r(y,s))&\quad&\text{if
}s\leq 1-4r^2,\\
&u(x,t)\geq C^{-1} u(\underline A_r(y,s))&&\text{if }s\geq -1+4r^2.
\end{alignat*}
\popQED
\end{lemma}

To state our next lemma, we need to use parabolic scaling operator on
$\R^{n}\times\R$. For any $(y,s)\in\R^n\times\R$ and $r>0$ we define
$$
T^r_{(y,s)}: (x,t)\mapsto \left(\frac{x-y}r,\frac{t-s}{r^2}\right). 
$$

\begin{lemma}[Localization property]\label{lem:loc-property} For $r\in
  (0,1/4)$ and $(y,s)\in\partial_p D$ and there exists a point
  $(\tilde y, \tilde s)\in \partial_p D\cap \Psi_{2r}(y,s)$ and $\tilde r\in [r,4r]$ such that
$$
\Psi_r(y,s)\cap D \subset \Psi_{\tilde r} (\tilde y, \tilde s)\cap D\subset \Psi_{8r}(y,s)\cap D
$$
and the parabolic scaling $T^{\tilde r}_{(\tilde y,\tilde
  s)}(\Psi_{\tilde r}(\tilde y,\tilde s)\cap D)$ is either

\begin{enumerate}
\item a rectangular box $\tilde\Psi$ such that
  $\Psi_{c_0}\subset\tilde \Psi \subset\Psi_{C_0}$
for some positive constants $c_0$ and $C_0$ depending on $L$ and $n$. 
\item union of two rectangular boxes as in \textup{(1)} with a common
  vertical side;
\item domain $\tilde D_{\tilde f}=\tilde\Psi\setminus E_f$
  with a thin Lipschitz complement at the end of Section~\ref{sec:domains-with-thin}.  
\end{enumerate}

\end{lemma}

\begin{proof} Consider the following cases:

1) $\Psi_{r}(y,s)\cap E_f=\emptyset$. In this case we take $(\tilde
y,\tilde s)=(y,s)$ and $\rho=r$. Then $\Psi_r(y,s)\cap
\Psi_1$ falls into category (1).

2) $\Psi_r(y,s)\cap E_f\not=\emptyset$, but $\Psi_{2r}(y,s)\cap
G_f=\emptyset$. In this case we take $(\tilde
y,\tilde s)=(y,s)$ and $\rho=2r$. In this case $\Psi_{2r}(y,s)\cap D$ splits
into the disjoint union of $\Psi_{2r}(y,s)\cap(\Psi_1)_\pm$ that falls
into category (2).

3) $\Psi_{2r}(y,s)\cap G_f\not=\emptyset$. In this case
choose $(\tilde y,\tilde s)\in \Psi_{3r}(y,s)\cap G_f$ with an
additional property $-1+r^2/4\leq \tilde s\leq 1-r^2/4$ and let $\rho=4r$.
Then $\Psi_\rho(\tilde y, \tilde s)\cap D =(\Psi_\rho(\tilde y,\tilde
s)\setminus E_f)\cap\Psi_1$ falls into category (3).
\end{proof}

\section{Regularity of $D$ for the heat equation}
\label{sec:barriers}

In this section we show that the domains $D$ with thin Lipschitz
complement $E_f$ are regular for the heat equation by using the existence of
an exterior thin cone at points on $E_f$ and applying Wiener-type
criterion for the heat equation \cite{EG}. Furthermore, we show the
existence of H\"older continuous local barriers at the points on
$E_f$, which we will use in the next section to prove the H\"older
continuity regularity of the solutions up to the parabolic boundary.

\subsection{PWB solutions}(\cite{Doob,Lieberman}) Given an open subset $\Omega \subset
\R^n\times\R$, let $\partial \Omega $ be its Euclidean
boundary. Define the parabolic boundary $\partial_p\Omega$ of $\Omega$
to be the set of all points $(x,t)\in \partial \Omega $ such that for any
$\epsilon >0$ the lower parabolic cylinder $Q_\epsilon(x,t)$ contains points not in
$\Omega $.  

We say that a function $u:\Omega \rightarrow (-\infty, +\infty]$ is
supercaloric if $u$ is lower semi-continuous, finite on dense subsets
of $\Omega$, and satisfies
the comparison principle in each parabolic cylinder $Q\Subset\Omega$: if $v\in C(\overline{Q})$ solves 
$\Delta v-\partial_t v=0$ in $Q$ and $v=u$ on $\partial _pQ$, then $v\leq u$ in $Q$.

A subcaloric function is defined as the negative of a supercaloric function. A function is caloric if it is supercaloric and subcaloric.

Given $g$, any real-valued function defined on $\partial _p\Omega $,
we define the upper solution 
\begin{align*}
\overline{H}_g=\inf\{u:u \text{ is supercaloric or identically } +\infty \text{ on each component of }\Omega,\\
 \liminf_{(y,s)\rightarrow (x,t)}u(y,s)\geq g(x,t) \text{ for all }(x,t)\in \partial _p\Omega, u \text{ bounded below on } \Omega \},
\end{align*}
and the lower solution 
\begin{align*}
\underline{H}_g=\sup\{u: u \text{ is subcaloric or identically } -\infty \text{ on each component of } \Omega,\\
 \limsup_{(y,s)\rightarrow (x,t)} u(y,s)\leq g(x,t) \text{ for all } (x,t)\in \partial _p\Omega, u \text{ bounded above on } \Omega \}.
\end{align*}
If $\overline{H}_g=\underline{H}_g$, then
$H_g=\overline{H}_g=\underline{H}_g$ is the Perron-Wiener-Brelot (PWB)
solution to the Dirichlet problem for $g$. It is shown in 1.VIII.4 and 1.XVIII.1 in \cite{Doob}
that if $g$ is a bounded continuous function, then the PWB solution $H_g$ exists and is unique for
any bounded domain $\Omega $ in $\R^n\times\R$.

Continuity of the PWB solution at points of $\partial _p\Omega $ is
not automatically guaranteed. A point $(x,t)\in \partial _p\Omega $ is a
regular boundary point if $\lim _{(y,s)\rightarrow (x,t)} H_g(y,s)=g(x,t)$ for
every bounded continuous function $g$ on $\partial _pD$. A necessary
and sufficient condition for a parabolic boundary point to be regular
is the existence of a local barrier for earlier time at that point
(Theorem~3.26 in \cite{Lieberman}). By a local barrier at
$(x,t)\in \partial_p \Omega$ we mean here a nonnegative continuous
function $w$ in $\overline{Q_r(x,t)\cap \Omega}$ for some $r>0$,
which has the following properties: 
(i) $w$ is supercaloric in $Q_r(x,t)\cap \Omega$;
(ii) $w$ vanishes only at $(x,t)$.

\subsection{Regularity of $D$ and barrier functions}\label{sec:self-similar-barrier} For the domain $D$ defined
in the introduction we have $\partial _p D=\partial _p\Psi_1\cup
E_f$. The regularity of $(x,t)\in \partial _p\Psi_1$ follows immediately
from the exterior cone condition for the Lipschitz domain. For $(x,t)\in
E_f$, instead of the full exterior cone we only know the existence of
a flat exterior cone centered at $(x,t)$ by the Lipschitz nature of the
thin graph. This will still be enough for the regularity, by the
Wiener-type criterion for the heat equation. We give the details below.

For $(x,t)\in E_f$, with parabolically Lipschitz $f$, there exist
$c_1,c_2>1$, depending on $n$ and $L$, such that
the exterior of $D$ contains a flat parabolic cone $\C(x,t)$ defined by
$$
\C(x,t)=(x,t)+\C
$$
$$
\C=\{(y,s)\in \R^n\times\R: s\leq 0,  y_{n-1}\leq
-c_1|y''|-c_2\sqrt{-s},y_{n}=0\}.
$$
Then by the Wiener-type criterion for the heat equation, established in \cite{EG}, the regularity
of $(x,t)\in E_f$ will follow once we show that
$$
\sum_{k=1}^\infty 2^{kn/2} \capacity(\A(2^{-k})\cap \C)=+\infty,
$$
where
$$
\A(c)=\{(y,s):(4\pi c)^{-n/2}\leq\Gamma(y,-s)\leq (2\pi
c)^{-n/2}\},
$$
$\Gamma$ is the heat kernel
\begin{align*}
\Gamma(y,s)=\left\{\begin{array}{ll}
(4\pi s)^{-n/2}e^{-|y|^2/4s}, & s>0,\\
0, & s\leq 0,
\end{array}\right.
\end{align*}
and $\capacity(K)$ is the thermal capacity for compact set $K$ 
defined by
\begin{multline*}
\capacity(K)=\sup\{\mu(K):\text{$\mu$ is a nonnegative Radon measure
    }\\\text{supported in $K$, s.t.\ $\mu*\Gamma\leq 1$ on $\R^n\times\R$}\}.
\end{multline*}
Because of the
self-similarity of $\C$, it is enough to verify that
$$
\capacity(\A(1)\cap\C)>0.
$$
The latter is easy to see, since we can take as $\mu$ the restriction
of $H^{n}$ Hausdorff measure to  $\A(1)\cap\C$ and note that
\begin{align*}
(\mu*\Gamma)(x,t)&=\int_{\A(1)\cap\C}\Gamma(x-y,t-s)dy'ds\\
&\leq \int_{-1}^0
\frac{1}{\sqrt{4\pi(t-s)^+}}ds\leq\int_{-1}^0
\frac{1}{\sqrt{4\pi(-s)}}ds<\infty
\end{align*}
for any $(x,t)\in\R^{n}\times\R$. Since $H^n(\A(1)\cap\C)>0$, we
therefore conclude that $\capacity(\A(1)\cap\C)>0$.
We therefore established the following fact.

\begin{proposition} The domain $D=D_f$ is regular for the heat equation.\qed
\end{proposition}

We next show that we can use the self-similarity of $\C$ to
construct a H\"older continuous barrier function at every $(x,t)\in
E_f$.

\begin{lemma}\label{lem:existence} There exist a nonnegative continuous function $U$ on
  $\overline{\Psi_1}$ with the following properties:
\begin{itemize} 
\item[(i)] $U>0$ in $\overline {\Psi_1}\setminus\{(0,0)\}$ and
  $U(0,0)=0$
\item[(ii)] $\Delta U-\partial_t U=0$ in $\Psi_1\setminus \C$.
\item[(iii)] $U(x,t)\leq C(|x|^2+|t|)^{\alpha/2}$ for $(x,t)\in
  \Psi_1$ and some $C>0$ and $0<\alpha<1$ depending only on $n$ and $L$.
\end{itemize}
\end{lemma}
\begin{proof} Let $U$ be a solution of the
  Dirichlet problem in $\Psi_1\setminus \C$ with boundary values
  $U(x,t)=|x|^2+|t|$ on $\partial_p(\Psi_1\setminus \C)$. Then $U$ will be
  continuous on $\overline \Psi_1$ and will satisfy the following properties:
\begin{itemize} 
\item[(i)] $U>0$ in $\overline {\Psi_1}\setminus\{(0,0)\}$ and $U(0,0)=0$;
\item[(ii)] $\Delta U-\partial_t U=0$ in $\Psi_1\setminus \C$.
\end{itemize}
In particular, there exists $c_0>0$ and $\lambda>0$ such that
$$
U\geq c_0\quad\text{on }\partial_p\Psi_1,\quad U\leq  c_0/2\quad\text{on
}\Psi_\lambda.
$$
We then can compare $U$ with its own parabolic scaling. Indeed, let
$M_U(r)=\sup_{\Psi_r} U$, for $0<r<1$. Then by the comparison principle for the heat
equation we will have
$$
U(x,t)\leq \frac{M_U(r)}{c_0} U(x/r,t/r^2),\quad\text{for }(x,t)\in
\Psi_r.
$$
(Carefully note that this inequality is
satisfied on $\C$ by the homogeneity of the boundary data on $\C$).
Hence, we will obtain that
$$
M_U(\lambda r)\leq \frac{ M_U(r)}{2},\quad\text{for any } 0<r<1,
$$
which will imply the H\"older continuity of $U$ at the origin by the
standard iteration. The proof is complete
\end{proof}

\section{Forward Boundary Harnack Inequalities}
\label{sec:forw-bound-harn}

In this section, we show the boundary H\"older regularity of the
solutions to the Dirichlet problem and follow the lines of
\cite{Kemper} to show the forward boundary Harnack inequality
(Carleson estimate).

We also need the notion of the caloric measure. Given a domain $\Omega
\subset \R^n\times\R$ and $(x,t)\in \Omega$, the caloric measure on
$\partial _p\Omega$ is denoted by $\omega ^{(x,t)}_\Omega$.  The
following facts about caloric measures can be found in
\cite{Doob}. For $B$ a Borel subset of $\partial _p\Omega$, we have
$\omega _\Omega ^{(x,t)}(B)=H_{\chi _B}(x,t)$, which is the PWB
solution to the Dirichlet problem 
\begin{equation*}
\Delta u - u_t = 0 \quad \text{in } \Omega; \quad u=\chi _B \text{ on } \partial _p\Omega,
\end{equation*}
where $\chi_B$ is the characteristic function of $B$. Given $g$ a bounded and continuous function on $\partial _p\Omega$, the PWB solution to the Dirichlet problem
\begin{equation*}
\Delta u-u_t=0 \text{ in }\Omega; \quad u=g \text{ on } \partial _p\Omega.
\end{equation*}
is given by $u(x,t)=\int_{\partial _p\Omega} g(y,s) d \omega ^{(x,t)}_\Omega (y,s)$. For a regular domain $\Omega$, one has the following useful property of caloric measures (\cite{Doob}):
\begin{proposition}
\label{prop:weaksol}
If $E$ is a fixed Borel subset of $\partial _p\Omega$, then the function $(x,t)\mapsto \omega ^{(x,t)}_\Omega(E)$ extends to $(y,s)\in \partial _p \Omega$ continuously provided $\chi _E$ is continuous at $(y,s)$.
\end{proposition}

\subsection{Forward boundary Harnack principle} From now on, we will
write the caloric measure with respect to $D=\Psi_1\setminus E_f$ as $\omega
^{(x,t)}$ for simplicity. Before we prove the forward boundary Harnack
inequality, we first show the H\"older continuity of the caloric
functions up to the boundary, which follows from the estimates on the barrier function constructed in Section~\ref{sec:barriers}. 

In what follows, for $0<r<1/4$ and $(y,s)\in\partial_pD$ we will
denote
$$
\Delta_r(y,s)=\Psi_r(y,s)\cap \partial_p D,
$$
and call it the \emph{parabolic surface ball} at $(y,s)$ of radius $r$.

\begin{lemma}
\label{lem:holder}
Let $0<r<1/4$ and $(y,s)\in \partial_p D$. Then there exist
$C=C(n,L)>0$ and $\alpha=\alpha(n,L) \in (0,1)$ such that if $u$ is
positive and caloric in $\Psi_r(y,s)\cap D$ and vanishes continuously on $\Delta _r(y,s)$, then
\begin{equation}\label{eq:holder}
u(x,t)\leq C\left(\frac{|x-y|^2+|t-s|}{r^2}\right)^{\alpha/2}M_u(r)
\end{equation}
for all $(x,t)\in \Psi_r(y,s)\cap D$, where $M_u(r)=\sup_{\Psi_r(y,s)\cap D}u$.
\end{lemma}
\begin{proof}
Let $U$ be the barrier function at $(0,0)$ in Lemma~\ref{lem:existence} and
$c_0=\inf_{\partial_p\Psi_1} U>0$. We then use the parabolic
scaling $T_{(y,s)}^r$ to construct a barrier function at $(y,s)$. If
$(y,s)\in \cN_r(E_f)$, then there is an exterior cone
$\C(y,s)$ at $(y,s)$ with a universal opening, depending only on $n,
L$, and 
$$
{U}_{(y,s)}^r:={U}\circ
T_{(y,s)}^r
$$
will be a local barrier function at $(y,s)$ and will satisfy 
\begin{equation}\label{eq:scale2}
0\leq
{U}_{(y,s)}^r(x,t)\leq C\left(\frac{|x-y|^2+|t-s|}{r^2}\right)^{\alpha/2},
\text{ for }(x,t)\in \Psi _r(y,s). 
\end{equation}
This construction can be made also at $(y,s)\in\partial_p
D\setminus \cN_r(E_f)$ as these points also have the exterior cone property and we
may still use the same formula for $U_{(y,s)}^r$, but after a
possible rotation of the coordinate axes in $\R^n$.

Then, by the maximum principle in $\Psi_r(y,s)\cap D$, we easily
obtain that 
\begin{equation}\label{eq:scale3}
u(x,t)\leq \frac{M_u(r)}{c_0}{U}_{(y,s)}^r(x,t),\quad \text{for }(x,t)\in \Psi_r(y,s)\cap D.
\end{equation}
Combining \eqref{eq:scale2} and \eqref{eq:scale3} we obtain \eqref{eq:holder}.
\end{proof}

The main result in this section is the following forward boundary
Harnack principle, also known as the Carleson estimate.
\begin{theorem}[Forward boundary Harnack principle or Carleson estimate]\label{thm:carleson} Let $r\in
  (0,1/4)$, $(y,s)\in \partial_p D$ with $s\leq 1-4r^2$, and $u$ be a
  nonnegative caloric function in $D$, continuously
vanishing on $\Delta _{3r}(y,s)$. Then there exists $C=C(n,L)>0$ such
that for $(x,t)\in \Psi_{r/2}(y,s)\cap D$
\begin{equation}\label{eq:carleson}
u(x,t)\leq C 
\begin{cases}
\max\{u(\overline{A}^+_r(y,s)),u(\overline{A}^-
_r(y,s))\},&\text{if }(y,s)\in \partial_p D\cap \cN_r(E_f)\\
u(\overline A_r(y,s)), & \text{if }(y,s)\in \partial_p D\setminus \cN_r(E_f)
\end{cases}
\end{equation}
\end{theorem}

To prove the Carleson estimate above, we need the following two lemmas
on the properties of the caloric measure in $D$, which correspond to
Lemmas 1.1 and 1.2 in \cite{Kemper}, respectively.

\begin{lemma}\label{lem:caloric1} For $0<r<1/4$, $(y,s)\in \partial_p D$
  with $s\leq 1-4r^2$, and
  $\gamma \in (0,1)$, there exists $C=C(\gamma, L)>0$ such that 
$$
\omega ^{(x,t)}(\Delta _r(y,s))\geq C,\quad\text{for }(x,t)\in \Psi _{\gamma
  r}(y,s)\cap D.
$$
\end{lemma}

\begin{proof} Suppose first $(y,s)\in \cN_r(E_f)$. Consider the
  caloric function $$
v(x,t):=\omega ^{(x,t)}_{\Psi_r(y,s)\setminus \C(y,s)}(\C(y,s)),
$$
where $\C(y,s)$ is the flat exterior cone defined in Section~\ref{sec:barriers}. The domain $\Psi_r(y,s)\setminus \C(y,s)$ is regular, hence by
Proposition \ref{prop:weaksol}, $v(x,t)$ is continuous on $\overline{\Psi_{\gamma r}(y,s)}$. We
next claim that there exists  $C=C(\gamma,n,L)>0$ such that
$$
v(x,t)\geq C\quad\text{in }\Psi_{\gamma r}(y,s).
$$
Indeed, consider the
normalized version of $v$
$$
v_0(x,t):=\omega ^{(x,t)}_{\Psi_1\setminus \C}(\C),
$$
which is related to $v$ through the identity
$v=v_0\circ T^r_{(y,s)}$. Then, from the continuity of $v_0$ in
$\overline{\Psi_\gamma}$, equality $v_0=1$ on $\C$, and the
strong maximum principle we obtain that $v_0\geq C=C(\gamma,n,L)>0$ on
$\overline{\Psi_{\gamma}}$. Using the parabolic scaling, we obtain
the claimed inequality for $v$. Moreover, applying comparison principle to $v(x,t)$ and $\omega ^{(x,t)}(\Delta _r(y,s))$ in $D\cap \Psi _r(y,s)$, we have
\[
\omega ^{(x,t)}(\Delta _r(y,s))\geq v(x,t)\geq C, \quad \text{for }
(x,t)\in D\cap \Psi _{\gamma r}(y,s).
\]
In the case when $(y,s)\in\partial_p D\setminus \cN_r(E_f)$, we may
modify the proof by changing the flat cone $\C(y,s)$ with the full cone
contained in the complement of $D$, or directly applying Kemper's
Lemma~1.1 in \cite{Kemper}.
\end{proof}

\begin{lemma}\label{lem:caloric2}
For $0<r<1/4$, $(y,s)\in \partial_p D$ with $s\leq 1-4r^2$, there
exists a constant $C=C(n,L)>0$, such that for any $r'\in (0,r)$ and
$(x,t)\in D\setminus \Psi _r(y,s)$, we have

\newlength\mywidth
\settowidth{\mywidth}{if $(y,s)\in \cN_r(E_f)$,}
\begin{equation}\label{eq:caloric2}
\omega^{(x,t)}(\Delta_{r'}(y,s))\leq C
\left\{
\begin{alignedat}{2}
&\omega^{\overline{A}_r(y,s)}(\Delta_{r'}(y,s)),&&\text{\hskip -\mywidth if
  $(y,s)\not\in \cN_r(E_f)$;}\\
&\max\{\omega ^{\overline{A}^+_r(y,s)}(\Delta_{r'}(y,s)), \omega
^{\overline{A}^-_r(y,s)}(\Delta_{r'}(y,s))\},\\
&&&\text{\hskip-\mywidth if $(y,s)\in \cN_r(E_f)$}.\\
\end{alignedat}
\right.
\end{equation}
\end{lemma}

\begin{proof}
For notational simplicity, we define
\begin{align*}
&\Delta ':=\Delta _{r'}(y,s), \quad \Delta := \Delta_r(y,s),\\
&\Psi^k:=\Psi_{2^{k-1}r'}(y,s),
\\
&\overline{A}_k^{\pm}:=\overline{A}_{2^{k-1}r'}^{\pm}(y,s),\quad\text{if
}\Psi^k\cap E_f\not=\emptyset\\
&\overline{A}_k:=\overline{A}_{2^{k-1}r'}(y,s),\quad\text{if
}\Psi^k\cap E_f=\emptyset\\
&\qquad\text{for }k=0,1,\ldots, \ell\quad \text{with}\quad 2^{\ell-1}r'<3r/4<2^{\ell}r'.
\end{align*}
We want to clarify here that for  $(y,s)\not\in E_f$ and small $r'$ and
$k$, it may happen that $\Psi^k$ does not intersect $E_f$. To be more
specific, let $\ell_0$ be the smallest
nonnegative integer such that $\Psi^{\ell_0}\cap E_f\not=
\emptyset$. Then we define $\overline A_k$ for $0\leq k\leq \min\{\ell_0-1,\ell\}$
and the pair $\overline A_k^\pm$ for $\ell_0\leq k\leq \ell$.

To prove the lemma, we want to show that there exists a universal constant $C$, in
particular independent of $k$, such that for $(x,t)\in D\setminus \Psi^k$
\begin{equation}\label{eq:Sk}\tag{$S_k$}
\omega^{(x,t)}(\Delta ')\leq C\begin{cases}\omega
  ^{\overline{A}_k}(\Delta '),&\text{if $1\leq k\leq
    \min\{\ell_0-1,\ell\}$},\\
\max\{\omega^{\overline{A}_k^+}(\Delta '),
\omega^{\overline{A}_k^-}(\Delta ')\}, & \text{if $\ell_0\leq
  k\leq \ell$}.
\end{cases}
\end{equation}
Once this is established, \eqref{eq:caloric2} will follow from $(S_{l})$ and the Harnack inequality.

The proof of $(S_k)$ is going to be by induction in $k$. We start with
an observation that by the Harnack inequality, there is $C_1>0$ independent of
$k$, $r'$ such that
\begin{equation}\label{eq:interiorharnack}
\begin{aligned}
\omega^{\overline A_k}(\Delta') &\leq C_1 \omega^{\overline
  A_{k+1}}(\Delta') && \text{for }0\leq k\leq \min\{\ell_0-2,\ell-1\}\\
\omega^{\overline A_{\ell_0-1}}(\Delta')&\leq C_1\max\{
\omega^{\overline A_{\ell_0}^{+}}(\Delta'),
\omega^{\overline{A}_{\ell_0}^{-}}(\Delta')\},&& \text{if }\ell_0\leq \ell\\
\omega^{\overline{A}_{k}^{\pm}}(\Delta ')&\leq C_1\omega
^{\overline{A}_{k+1}^{\pm}}(\Delta'),&& \text{for } \ell_0\leq k\leq \ell-1.
\end{aligned}
\end{equation}

\medskip\noindent\emph{Proof of $(S_1)$}: Without loss of
generality assume $(y,s)\in \partial_p D\cap \overline D_+$.

Case 1) Suppose first that $\Psi^1\cap E_f=\emptyset$, i.e., $\ell_0>1$. 
In this case
$\overline A_{0}=\overline A_{r'/2}(y,s)\in
\Psi_{(3/4)r'}(y,s)$ and by Lemma~\ref{lem:caloric1}
there exists a universal $C_0>0$, such that $\omega ^{\overline
  A_{0}}(\Delta ')\geq C_0$. By \eqref{eq:interiorharnack} we
have $\omega ^{\overline A_{0}}(\Delta ')\leq C_1\omega
^{\overline{A}_1}(\Delta ')$. Letting $C_2=C_1/C_0$, we then have 
\begin{equation}
\omega ^{(x,t)}(\Delta ')\leq 1\leq C_2\omega ^{\overline{A}_1}(\Delta ').
\end{equation}

\medskip
Case 2) Suppose now, $\Psi^1\cap E_f\not=\emptyset$, but $\Psi^0\cap
E_f=\emptyset$, i.e., $\ell_0=1$. In this case we start as in Case 1)
and finish by applying the second inequality in
\eqref{eq:interiorharnack}, which yields

\begin{equation}
\omega ^{(x,t)}(\Delta ')\leq 1\leq C_2\max\{\omega
^{\overline{A}_1^+}(\Delta '),\omega
^{\overline{A}_1^-}(\Delta ')\}.
\end{equation}

Case 3) Finally, assume that  $\Psi^0\cap E_f\not=\emptyset$, i.e., $\ell_0=0$. 
Without loss of generality
assume also that $(y,s)\in \partial_p D\cap \overline D_+$. In this case $\overline A_{0}^+\in
\Psi_{(3/4)r'}(y,s)$ and therefore $\omega^{\overline A_0^+}(\Delta')\geq
C_0$. Besides, by \eqref{eq:interiorharnack}, we have that
$\omega^{\overline A_0^+}(\Delta')\leq
  C_1\omega^{\overline A_1^+}(\Delta')$, which yields
\begin{equation}
\omega^{(x,t)}(\Delta ')\leq 1\leq C_2\omega ^{\overline{A}_1^+}(\Delta ').
\end{equation}
This proves $(S_1)$ with the constant $C=C_2$.
 
\medskip
We now turn to the proof of the induction step.

\medskip\noindent\emph{Proof of $(S_k)\Rightarrow
  (S_{k+1})$}: More
precisely, we will show that if $(S_k)$ holds with some universal
constant $C$ (to be specified) then $(S_{k+1})$ also holds with
the same constant.

By the maximum principle, we need to verify
$(S_{k+1})$ for  $(x,t)\in \partial_p (D\setminus \Psi^{k+1})$. Since
$\omega^{(x,t)}(\Delta')$ vanishes on $(\partial_p D)\setminus
\Psi^{k+1}$, we may assume that $(x,t)\in (\partial \Psi^{k+1})\cap D$. 
We will need to consider three cases, as in the proof of $(S_1)$: 

1) $\Psi^{k+1}\cap E_f=\emptyset$, i.e., $\ell_0>k+1$; 

2) $\Psi^{k+1}\cap E_f\not=\emptyset$, 
but $\Psi^k\cap E_f=\emptyset$, i.e., $\ell_0=k+1$;

3) $\Psi^{k}\cap E_f\not=\emptyset$, i.e., $\ell_0\leq k$. 

Since the proof is similar in all three cases, we will treat only Case 2) in detail.

\medskip
Case 2) So suppose  $\Psi^{k+1}\cap E_f\not=\emptyset$
but $\Psi^k\cap E_f=\emptyset$. We consider two
subcases, depending weather $(x,t)\in \partial\Psi^{k+1}$ is close to
$\partial_p D$ or not.

\smallskip
Case 2a) First assume that $(x,t)\in \cN_{\mu 2^kr'}(\partial_p D)$ for some small positive $\mu=\mu(L,n)<1/2$ (to be specified). Take $(z,h)\in \Psi_{\mu
  2^k r'}(x,t)\cap \partial_p D$ and observe that $\omega^{(x,t)}(\Delta')$ is
caloric in $\Psi_{2^{k-1}r'}(z,h)\cap D$ and vanishes continuously on
$\Delta _{2^{k-1}r'}(z,h)$ (by
Proposition~\ref{prop:weaksol}). Besides, by the induction assumption
that $(S_k)$ holds, we have
$$
\omega^{(x,t)}(\Delta')\leq C \omega^{\overline A_{k}}(\Delta'),\quad
\text{for }(x,t)\in \Psi_{2^{k-1}r'}(z,h)\cap D\subset D\setminus \Psi^k.
$$
Hence,  by Lemma~\ref{lem:holder}, if
$\mu=\mu(n,L)>0$ is small enough, we obtain that
$$
\omega^{(x,t)}(\Delta')\leq \frac1{C_1} C \omega ^{\overline{A}_k}(\Delta ')
,\quad\text{for }(x,t)\in \Psi_{\mu 2^kr'}(z,h).
$$
Here $C_1$ is the constant in \eqref{eq:interiorharnack}.
This, combined with \eqref{eq:interiorharnack}, gives
\begin{align*}
\omega^{(x,t)}(\Delta')&\leq \frac{C}{C_1}\omega^{\overline{A}_k}(\Delta')\\
&\leq \frac{C}{C_1}\cdot C_1\max\{\omega^{\overline{A}_{k+1}^+}(\Delta'),\omega^{\overline{A}_{k+1}^-}(\Delta') \}\\
&=C \max\{\omega^{\overline{A}_{k+1}^+}(\Delta'),\omega^{\overline{A}_{k+1}^-}(\Delta') \}.
\end{align*}
This proves $(S_{k+1})$ for $(x,t)\in \cN_{\mu2^k r'}(\partial_p
D)\cap \partial \Psi^{k+1}$.

\smallskip
Case 2b) Assume now $\Psi_{\mu 2^k r'}(x,t)\cap \partial_p D=\emptyset$. In
this case, it is easy to see that we can construct a parabolic Harnack
chain in $D$ of universal length from $(x,t)$ to either $\overline
A_{k+1}^+$ or $\overline
A_{k+1}^-$,  which implies that for some universal constant $C_3>0$
\begin{equation*}\label{eq:beta22}
\omega ^{(x,t)}(\Delta')\leq
C_3\max\{\omega^{\overline{A}_{k+1}^+}(\Delta'),\omega^{\overline{A}_{k+1}^-}(\Delta')\}.
\end{equation*}
Thus, combing Cases 2a) and 2b), we obtain that $(S_{k+1})$ holds with 
provided $C=\max\{C_2,C_3\}$. This completes the proof of our
induction step in Case 2). As we mentioned earlier, Cases 1) and 3)
are obtained by a small modification from Case 1) as in the proof of
$(S_1)$.
This completes the proof of the lemma.
\end{proof}

Now we prove the Carleson estimate. With Lemma~\ref{lem:caloric1} and
Lemma~\ref{lem:caloric2} at hand, we use ideas similar to those in \cite{Salsa}. 

\begin{proof}[Proof of Theorem~\ref{thm:carleson}] We start with a
  remark that if $(y,s)\not\in\cN_{r/4}(E_f)$ then we can restrict $u$
  to $D_+$ or $D_-$ and obtain the second estimate in \eqref{eq:carleson}
  from the known result for parabolic Lipschitz domains. We thus consider only the case $(y,s)\in\cN_{r/4}(E_f)$. Besides,
  replacing $(y,s)$ with $(y',s')\in \Psi_{r/4}(y,s)\cap E_f$ we may
  further assume that $(y,s)\in E_f$, but then we will need to change
  the assumption that $u$ vanishes on $\Delta_{2r}(y,s)$ and prove the estimate
  \eqref{eq:carleson} for $(x,t)\in \Psi_{r}(y,s)\cap D$.

With the above assumptions in mind, let $0<r<1/4$ and $R=8r$. Let $\tilde
D_{R}(y,s):=\Psi_{\tilde R}(\tilde{y},\tilde{s})\cap D$ be given by the localization property
Lemma~\ref{lem:loc-property}. Note that we will be either in case (2)
or (3) of that lemma, moreover, we can choose $(\tilde y,\tilde
s)=(y,s)$.

For the notational brevity, let $\omega_R^{(x,t)}:=\omega _{\tilde D_{R}(y,s)}^{(x,t)}$ be the caloric
measure with respect to $\tilde D_{R}(y,s)$. We will also skip the
center $(y,s)$
in the notations $\tilde D_{R}(y,s)$ for $\Psi_\rho(y,s)$ and $\Delta_\rho(y,s)$.

Since $u$ is caloric in $\tilde D_R$ and continuously vanishes up to
$\Delta_{2r}$, we have
\begin{equation}\label{eq:carleson1}
u(x,t)=\int_{(\partial_p \tilde
  D_R)\setminus\Delta_{2r}}u(z,h)d\omega_R^{(x,t)}(z,h),\quad
(x,t)\in \tilde D_R.
\end{equation}
Note that for $(x,t)\in \Psi_r\cap D$, we have $(x,t)\not\in
\Psi_{r/2}(z,h)$ for any
$(z,h)\in (
\partial_p\tilde D_R)\setminus\Delta_{2r}$. Hence,
applying Lemma~\ref{lem:caloric2}\footnote{We have to scale the domain
  $\tilde D_R$
with $T^{\tilde R}_{(\tilde y,\tilde s)}$ first and apply  Lemma
\ref{lem:caloric2} to $r/2\tilde R<1/8$ if we are in case (3) of the
localization property Lemma~\ref{lem:loc-property}; in the case
(2) we apply the known results for parabolic Lipschitz domains.} to $\omega_R^{(x,t)}$ in
$\tilde D_R$, we will have that for
$(x,t)\in \Psi_{r}\cap D$ and sufficiently small $r'$ 
\begin{multline*}
\omega_R^{(x,t)}(\Delta_{r'}(z,h))\leq C\max \big\{ \omega_R^{\overline{A}^+_{r/2,R}(z,h)}(\Delta_{r'}(z,h)), \omega_R^{\overline{A}^-_{r/2,R}(z,h)}(\Delta_{r'}(z,h))\big\}\\
\text{for } (z,h)\in \cN_{r/2}(E_f)\cap
(\partial_p\tilde D_R)\setminus \Delta_{2r}
\end{multline*}
and 
\begin{multline*}
\omega^{(x,t)}_R(\Delta _{r'}(z,h))\leq C\,\omega
^{\overline{A}_{r/2,R}(z,h)}_R(\Delta _{r'}(z,h)),\\
 \text{for } (z,h)\in \partial_p \tilde D_R\setminus
 (\cN_{r/2}(E_f)\cup \Delta_{2r}),
\end{multline*}
where $C=C(L,n)$ and by $\overline A_{r/2,R}^\pm$ and $\overline A_{r/2,R}$ we denote
the corkscrew points with respect to the domain $\tilde D_R$.
To proceed, we note that for $(z,h)\in \partial_p\tilde D_R$ with $h>s+r^2$, by the maximum
principle 
$$
\omega_R^{(x,t)}(\Delta_{r'}(z,h))=0
$$ for any
$(x,t)\in \Psi_{r}\cap D$ provided $r'$ is small enough.
For $(z,h)\in (\partial_p\tilde D_R)\setminus\Delta_{2r}$ with
$h\leq s+r^2$, we note that with the help of Lemmas
\ref{lem:H-chain-prop-I} and \ref{lem:H-chain-prop-II} we can
construct a Harnack chain of controllable length in $D$ from
$\overline A_{r/2,R}^\pm(z,h)$ or $\overline A_{r/2,R}(z,h)$ to
$\overline A_{r}^+(y,s)$ or $\overline A_{r}^-(y,s)$ (corkscrew
points with respect to the original $D$). This will imply that for
$(x,t)\in \Psi_r\cap D$
and $(z,h)\in\partial_p\tilde D_R\setminus \Delta_{2r}$
\begin{multline}\label{eq:caloric5}
\omega_R^{(x,t)}(\Delta_{r'}(z,h))\leq C\max
\{\omega_R^{\overline{A}^+_r(y,s)}(\Delta_{r'}(z,h)),
\omega_R^{\overline{A}^-_r(y,s)}(\Delta_{r'}(z,h))\}.
\end{multline}
We now want to apply Besicovitch's theorem on the differentiation of
Radon measures. However, since $\partial_p \tilde D_R$ locally is
not topologically  equivalent to a Euclidean space, we make the
following symmetrization argument. 
For $x\in\R^n$ let $\hat x$ be its mirror image with respect to the
hyperplane $\{x_{n}=0\}$. 
We then can write
\begin{align*}
&u(x,t)+u(\hat x,t)=\int_{\partial_p \tilde
  D_R\setminus\Delta_{2r}} [u(z,h)+u(\hat z,h)]
d\omega^{(x,t)}_R(z,h)\\
&\qquad=\frac12\int_{\partial_p \tilde
  D_R\setminus\Delta_{2r}} [u(z,h)+u(\hat z,h)]
\left(d\omega^{(x,t)}_R(z,h)+d\omega^{(\hat x,t)}_R(z,h)\right)\\
&\qquad=\int_{\partial_p ((\tilde D_R)_+)\setminus
  \Delta_{2r}}[u(z,h)+u(\hat z,h)]\chi
\left(d\omega^{(x,t)}_R(z,h)+d\omega^{(\hat x,t)}_R(z,h)\right),
\end{align*}
where $\chi=1/2$ on $\partial_p ((\tilde D_R)_+)\cap\{x_{n}=0\}$
and $\chi=1$ on the remaining part of $\partial_p ((\tilde D_R)_+)$
and the measures $d\omega^{(x,t)}_R$ and $d\omega^{(\hat x,t)}_R$ are extended as
zero on the thin space outside $E_f$, i.e., on $\partial_p ((\tilde
D_R)_+)\setminus \partial_p \tilde D_R$. We then use
the estimate \eqref{eq:caloric5} for $(x,t)$ and $(\hat x,t)$ in
$\Psi_{r}\cap D$. Now note that in this situation we can apply Besicovitch's theorem on
differentiation, since we can locally project $\partial_p ((\tilde
D_R)_+)$ to hyperplanes, similarly to \cite{HW}. This will yield
\begin{equation}\label{eq:caloric6}
\frac{d\omega_R^{(x,t)}(z,h)+d\omega_R^{(\hat x,t)}(z,h)}{d\omega_R^{\overline{A}^{+}_r(y,s)}(z,h)+d\omega_R^{\overline{A}^{-}_r(y,s)}(z,h)}\leq C
\end{equation}
for $(z,h)\in\partial_p((\tilde D_R)_+)\setminus\Delta_{2r}$
and $(x,t)\in \Psi_{r}\cap D$.
Hence, 
we obtain
\begin{align*}
&u(x,t)+u(\hat x,t)\\&\qquad\leq C\int_{\partial_p ((\tilde
  D_R)_+)\setminus \Delta_{2r}}[ u(z,h)+u(\hat z,h)]\left(d\omega_R^{\overline{A}^+_r(y,s)}(z,h)+d\omega_R^{\overline{A}^-_r(y,s)}(z,h)\right)\\
&\qquad\leq C\left(u(\overline{A}^+_r(y,s))+u(\overline{A}^-_r(y,s))\right),\\
&\qquad\leq C\max\{u(\overline{A}^+_r(y,s)), u(\overline{A}^-_r(y,s))\}, \quad (x,t)\in \Psi_r\cap D.
\end{align*}
This completes the proof of the theorem.
\end{proof}

The following theorem is a useful consequence of Theorem
\ref{thm:carleson}, whose proof is similar to that of Theorem~1.1 in \cite{Garofalo2} with Theorem~\ref{thm:carleson} above in hand. Hence here we only state the theorem without giving a proof.

\begin{theorem}\label{thm:carleson2} For $0<r<1/4$,
  $(y,s)\in \partial _pD$ with $s\leq 1-4r^2$, let $u$ be caloric in
  $D$ and continuously vanishes on $\partial _pD\setminus
  \Delta_{r/2}(y,s)$. Then there exists $C=C(n,L)$ such that for
  $(x,t)\in D\setminus \Psi_r(y,s)$ we have
\begin{equation}\label{eq:carleson2}
u(x,t)\leq C\begin{cases}
\max\{u(\overline{A}_r^+(y,s)),
u(\overline{A}_r^-(y,s))\},&\text{if }(y,s)\in\cN_r(E_f)\\
u(\overline{A}_r(y,s)), &\text{if }(y,s)\not\in\cN_r(E_f).
\end{cases}
\end{equation}
Moreover, applying Lemma~\ref{lem:caloric1} and the maximum principle we have: for $(x,t)\in D\setminus \Psi_r(y,s)$,
\begin{multline}\label{eq:carleson3}
u(x,t)\leq C\omega^{(x,t)}(\Delta _{2r}(y,s))\times\\ \times\begin{cases}
\max\{u(\overline{A}_r^+(y,s)),
u(\overline{A}_r^-(y,s))\},&\text{if }(y,s)\in\cN_r(E_f)\\
u(\overline{A}_r(y,s)), &\text{if }(y,s)\not\in\cN_r(E_f).
\end{cases}
\end{multline}
\end{theorem}

\section{Kernel functions}
\label{sec:kernel-functions}

Before proceeding to the backward boundary Harnack principle, we need
the notion of kernel functions associated to the heat operator and the
domain $D$. In \cite{Garofalo2}, the backward Harnack principle is a
consequence of the global comparison principle (Theorem
\ref{thm:global}) by a simple time-shifting argument. In our case,
since $D$ is not cylindrical, the above simple argument does not
work. So we will first prove some properties of the kernel functions
which can be used to show the doubling property of the caloric
measures as in \cite{Wu}. Then, using arguments as in
\cite{Garofalo2}, we obtain the the backward Harnack principle.

\subsection{Existence of kernel functions} 
\label{sec:exist-kern-funct}

Let $(X, T)\in D$ be
fixed. Given $(y,s)\in \partial _pD$ with $s<T$, a function
$K(x,t;y,s)$ defined in $D$ is called a kernel function at $(y,s)$ for the heat equation with respect to $(X,T)$ if,
\begin{itemize}
   \item[(i)]$K(\cdot,\cdot;y,s)\geq 0$ in $D$,
   \item[(ii)]$(\Delta-\partial _t)K(\cdot,\cdot;y,s)=0$ in $D$,
   \item[(iii)]${\displaystyle \lim_{\substack{(x,t)\rightarrow (z,h)\\(x,t)\in D}}K(x,t; y,s)=0}$ for $(z,h)\in \partial _pD\setminus\{(y,s)\}$,
   \item[(iv)]$K(X,T;y,s)=1$.
 \end{itemize}
If $s\geq T$, $K(x,t;y,s)$ will be taken identically equal to zero. We note that by maximum principle $K(x,t;y,s)=0$ when $t<s$.

The existence of the kernel functions (for the heat operator on domain $D$) follows directly from Theorem~\ref{thm:carleson}. Let $(y,s)\in \partial _pD$ with $s<T-\delta^2$ for some $\delta>0$, consider
\begin{equation}\label{eq:kernel}
v_n(x,t)=\frac{\omega^{(x,t)}(\Delta _{\frac{1}{n}}(y,s))}{\omega ^{(X,T)}(\Delta _{\frac{1}{n}}(y,s))},\quad (x,t)\in D, \quad \frac{1}{n}<\delta.
\end{equation}
We clearly have $v_n(x,t)\geq 0$, $(\Delta -\partial _t)v_n(x,t)=0$ in
$D$ and $v_n(X,T)=1$. Given $\epsilon\in (0,1/4)$ small, by Theorem
\ref{thm:carleson2} and the Harnack inequality $\{v_n\}$ is uniformly
bounded on $\overline{D\setminus \Psi_\epsilon(y,s)}$ if $n\geq
2/\epsilon$. Moreover, by the up to the boundary regularity
(see Proposition~\ref{prop:weaksol} and Lemma~\ref{lem:holder}), the family
$\{v_n\}$ 
is uniformly H\"older in $\overline{D\setminus \Psi _\epsilon
  (y,s)}$. Hence, up to a subsequence, $\{v_n\}$ converges uniformly on
$\overline{D\setminus \Psi _\epsilon (y,s)}$ to some nonnegative caloric function $v$ satisfying $v(X,T)=1$.
Since $\epsilon$ can be taken arbitrarily small, $v$ vanishes on $\partial
_pD\setminus \{(y,s)\}$. Therefore, $v(x,t)$
is a kernel function at $(y,s)$.  

\theoremstyle{definition}
\newtheorem{convention}[theorem]{Convention}
\begin{convention}\label{convention}
From now on, to avoid cumbersome details we will make a time extension
of domain $D$ for $1\leq t<2$ by looking at
$$
\tilde D=\tilde \Psi\setminus E_f,\quad \tilde\Psi=(-1,1)^n\times(-1,2)
$$
as in Section~\ref{sec:domains-with-thin}. We then fix $(X,T)$ with
$T=3/2$ and $X\in \{x_{n}=0\}$, $X_{n-1}>3nL$ and normalize all kernels
$K(\cdot, \cdot; \cdot, \cdot)$ at this point $(X,T)$. In this way we
will be able to state the results in this section for our original
domain $D$.  Alternatively, we could fix $(X,T)\in D$, and then state
the results in the part of the domain $D\cap
\{(x,t):-1<t<T-\delta^2\}$ with some
$\delta>0$, with the additional dependence of constants on $\delta$.
\end{convention}
\subsection{Nonuniqueness of kernel functions at $E_f\setminus G_f$}

The idea is, if we consider the completion $D^*$ of domain $D$ with
respect to the inner metric $\rho_D$ and let $\partial^* D=
D^*\setminus D$, then it is clear that each Euclidean boundary point $(y,s)
\in G_f$ and $(y,s)\in \partial_p\Psi_1$ will correspond to only one $(y,s)^*
\in \partial^*D$, and each $(y,s)\in E_f\setminus G_f$ will
correspond to exactly two points $(y,s)_+^*$,
$(y,s)_-^*\in \partial^* D$. It is not hard to imagine
that the kernel functions corresponding to $(y,s)_+^*$ and
$(y,s)_-^*$ are linearly independent and they are the two
linearly independent kernel functions at $(y,s)$. In this section we will make this idea precise by considering
two-sided caloric measures $\vartheta_+$ and $\vartheta_-$. We will study the
properties of $\vartheta_+$ and $\vartheta_-$ and their relationship with the caloric
measure $\omega_D$.

First we introduce some more notations. Given $(y,s)\in \partial_pD\setminus G_f$, let 
\begin{equation}\label{eq:r0}
r_0=\sup \{r\in (0,1/4):\Delta_{2r}(y,s)\cap G_f=\emptyset\}.
\end{equation}
Note that $r_0$ is a constant depending on $(y,s)$ and is such that
for any $0<r<r_0$, $\Psi_{2r}(y,s)\cap D$ is either separated by $E_f$
into two disjoint sets $\Psi_{2r}^+$ and $\Psi_{2r}^-$ or $\Psi_{2r}(y,s)\cap D \subset D_+$ (or $D_-$). We define for
$0<r<r_0$ the following shifting operators $F_r^+$ and $F_r^-$:
\begin{align}
F_r^+(x,t)& =(x'', x_{n-1}+4nL r, x_{n}+r, t+4r^2),\label{eq:shift11}\\
F_r^-(x,t)& =(x'', x_{n-1}+4nL r, x_{n}-r, t+4r^2).\label{eq:shift22}
\end{align}
For any $0<r<r_0$, define \begin{equation}\label{eq:shift00}
D_r^+=D\setminus (E_{r,1}^+\cup E_{r,2}^+\cup E_{r,3}^+\cup E_{r,4}^+),
\end{equation}
where 
\begin{align*}
E_{r,1}^+& :=\{(x,t)\in \R^n\times\R: x_{n-1}\leq f(x'',t),-r\leq x_{n}\leq 0\},\\
E_{r,2}^+& :=\{(x,t):1-r\leq x_{n}\leq 1\},\\
E_{r,3}^+& :=\{(x,t):4nL(1-r)\leq x_{n-1}\leq 4nL\},\\
E_{r,4}^+& :=\{(x,t):1-4r^2\leq t\leq 1\}.
\end{align*}
It is easy to see that $D_r^+\subset D$ and $F_r^+(D_r^+)\subset D$.
Similarly we can define $D_r^-\subset D$ satisfying
$F_r^-(D_r^-)\subset D$. Notice that $D^+_r\nearrow D$, $D^-_r\nearrow
D$ as $r\searrow 0$. Moreover, it is clear that for each $r\in (0,1/4)$
\begin{align}
& \cN_{1/4}(E_f)\cap \partial_pD\subset (\partial _pD_r^+\cup \partial
_pD_r^-)\cap \partial_p D, \label{eq:add1}\\
& E_f\subset \partial_pD_r^+\cap \partial_pD^-_r. \label{eq:add2}
\end{align}

Let $\omega ^{+}_r$ and $\omega ^{-}_r$ denote the caloric measures
with respect to $D_r^+$ and $D_r^-$ respectively. Given $(x,t)\in D$ and $r>0$ small enough such that $(x,t)\in D^+_r\cap D^-_r$,
$\omega ^{\pm^{(x,t)}}_r$ are Radon measures
on $\partial_p( D^\pm_r)\cap \partial _p ( D_\pm)$ (recall $D_+ (D_-)= D\cap \{x_n>0 (<0)\}$). Moreover, let $K$
be a relatively compact Borel subset of $\partial_p( D^\pm_r)\cap \partial _p ( D_\pm)$, by the
comparison principle
$\omega^{\pm^{(x,t)}}_r(K)\leq\omega^{\pm^{(x,t)}}_{r'}(K)$ for
$0<r'<r$. Hence there exist Radon measures $\vartheta_\pm^{(x,t)}$ on $\partial_p( D^\pm_r)\cap \partial _p ( D_\pm)$, such that
$$\omega ^{\pm^{(x,t)}}_r|_{\partial_p( D^\pm_r)\cap \partial _p ( D_\pm)}\stackrel{\ast}{\rightharpoonup} \vartheta_\pm^{(x,t)},  \quad r\rightarrow
0.$$
For $(y,s)\in (\cN_{1/4}(E_f)\cap \partial_p D)\setminus
G_f$ and $0<r<r_0$, denote
$$\Delta_r^\pm (y,s) :=\Delta_r(y,s)\cap \partial_p D_\pm, \quad \text{if } \Delta_r(y,s)\cap \partial_p(D_\pm)\neq \emptyset .$$
Note that if $\Delta_r(y,s)\subset E_f$, then $\Delta_r^\pm (y,s)=\Delta_r(y,s)$.
It is easy to see that $(x,t)\mapsto \vartheta_{\pm}^{(x,t)}(\Delta
_r^\pm(y,s))$ are caloric in $D$.

To simplify the notations we will write $\Delta _r$, $\Delta _r^\pm$ instead of $\Delta
_r(y,s)$, $\Delta_r^\pm (y,s)$. If $\Delta_r(y,s)\cap \partial_p(D_+)$ (or $\Delta_r(y,s)\cap \partial_p(D_-)$) is empty, we set $\vartheta_+^{(x,t)}(\Delta_r^+(y,s))=0$ (or $\vartheta_-^{(x,t)}(\Delta_r^-(y,s))=0$).

We also note that with Convention~\ref{convention} in mind, the future
corkscrew points $\overline A_{r}^\pm(y,s)$ or $\overline A_{r}(y,s)$, $0<r<r_0$
and are defined for all $s\in[-1,1]$.

\begin{proposition}\label{prop:limitfunc}
Given $(y,s)\in (\cN_{1/4}(E_f)\cap \partial_pD)\setminus G_f$, for $0<r<r_0$ we have,
\begin{itemize}
\item[(i)] $$\sup_{(x,t)\in \partial_p D_{r'}^+\cap D}\vartheta^{(x,t)}_+(\Delta ^+ _r), \quad \sup_{(x,t)\in \partial _pD_{r'}^-\cap D}\vartheta^{(x,t)}_-(\Delta ^-_r)\rightarrow 0, \text{ as }r'\rightarrow 0.$$
\item[(ii)] $\vartheta^{(x,t)}_+(\Delta _r^+)+\vartheta^{(x,t)}_-(\Delta _r^-)=\omega ^{(x,t)}(\Delta _r)$, for $(x,t)\in D$.\\
\item[(iii)] There exists a constant $C=C(n,L)$ such that for any $0<r'<r$
\begin{eqnarray*}
\vartheta^{(x,t)}_+(\Delta _{r'}^+)\leq C\vartheta_+^{\overline{A}_r^+(y,s)}(\Delta _{r'}^+)\vartheta^{(x,t)}_+(\Delta _{2r}^+),\quad \text{for}\ (x,t)\in D\setminus\Psi_r^+(y,s),\\
\vartheta^{(x,t)}_-(\Delta _{r'}^-)\leq C\vartheta_-^{\overline{A}_r^-(y,s)}(\Delta _{r'}^-)\vartheta^{(x,t)}_-(\Delta _{2r}^-),\quad \text{for}\ (x,t)\in D\setminus\Psi_r^-(y,s).
\end{eqnarray*}
\item[(iv)] For $(X,T)$ as defined above and $(y,s)\in E_f\setminus G_f$, there exists a positive constant $C=C(n,L,r_0)$ such that
\begin{equation*}
C^{-1} \vartheta^{(X,T)}_+(\Delta _{r}^+) \leq \vartheta^{(X,T)}_-(\Delta _{r}^-)\leq C \vartheta^{(X,T)}_+(\Delta _{r}^+).
\end{equation*}
\end{itemize}
\end{proposition}

\begin{proof}
\step{Proof of (i)}:
We assume that $\Delta_r^\pm \not=\emptyset$. If either of them is empty, the conclusion holds obviously.

For $0<r<r_0$ we have 
\begin{eqnarray*}
\partial _pD_{r}^+\cap D=\{(x,t)\in D:x_{n-1}=4nL(1-r) \text{ or
}x_{n}=1-r\}\cup\\ 
\{(x,t)\in D: x_{n-1}\leq f(x'',t),x_{n}=-r\text{ or } x_{n-1}=f(x'',t),-r\leq
x_{n}<0\}. 
\end{eqnarray*}
Given $(y,s)\in (\cN_{1/4}(E_f)\cap \partial_pD)\setminus G_f$, let $0<r''<r'<r_0$, then $\omega
^{+^{(x,t)}}_{r''}(\Delta _r^+(y,s))$ is caloric in $D^+_{r''}$ and from
the way $r_0$ is chosen vanishes continuously on $\Delta _{r_0}(z,h)$
for each $(z,h)\in \partial _p D_{r''}^+ \cap D$. Notice that 
$$
\partial _pD_{r'}^+ \cap D \subset \bigcup _{(z,h)\in \partial
  _pD_{r''}^+\cap D}\Psi _{r_0}(z,h),
$$ 
hence applying Lemma~\ref{lem:holder} in each $\Psi _{r_0}(z,h)\cap
D_{r''}^+$ we obtain constants $C=C(n,L)$ and $\gamma =\gamma(n,L)$,
$\gamma\in (0,1)$ such that 
\begin{equation}\label{eq:last}
\omega^{+^{(x,t)}}_{r''}(\Delta _r^+)\leq C\left(\frac{|x-z|+|t-h|^{\frac{1}{2}}}{r_0}\right)^{\gamma}\leq C \left(\frac{r'}{r_0}\right)^{\gamma},\quad \forall (x,t)\in \partial _pD_{r'}^+\cap D.
\end{equation}
The constant $C$ and $\gamma $ above do not depend on $(z,h)\in \partial _pD_{r''}^+\cap D$, $r$ or $r''$ because of the existence of the exterior flat parabolic cones centered at each $(z,h)$ with an uniform opening depending only on $n$ and $L$ .

Let $r''\rightarrow 0$ in \eqref{eq:last}, then we get
$$\vartheta_+^{(x,t)}(\Delta _r^+)\leq C \left(\frac{r'}{r_0}\right)^{\gamma}, \quad \text{ uniformly for } (x,t)\in \partial _pD_{r'}^+\cap D.$$
Therefore,
$$\lim _{r'\rightarrow 0}\sup _{(x,t)\in \partial _pD_{r'}^+\cap D}\vartheta_+^{(x,t)}(\Delta _r^+)=0, $$
which finishes the proof.

\step{Proof of (ii)}: Let $\chi _{\Delta _r}$ be the characteristic
function of $\Delta _r$ on $\partial _pD$. Let $g_n$ be a sequence of nonnegative continuous functions on $\partial _p D$ such that
 $g_n\nearrow \chi _{\Delta _r}$. 
Let $u_n$ be the solution to the heat equation in $D$ with
boundary values $g_n$. Then by the maximum principle, $u_n(x,t)\nearrow
\omega ^{(x,t)}(\Delta _r)$ for $(x,t)\in D$. 

Now we estimate $\vartheta_+^{(x,t)}(\Delta _r^+)+\vartheta_-^{(x,t)}(\Delta
_r^-)$. Let $u_{n,r'}^+(x,t)$ be the solution to the heat equation in
$D_{r'}^+$ with boundary value equal to $g_n$ on $\partial
_pD_{r'}^+\cap \partial _pD$ and equal to $\vartheta_+^{(x,t)}(\Delta
_r^+)$ otherwise. Since
$\vartheta_+^{(x,t)}(\Delta _r^+)=\lim_{r''\rightarrow 0}\omega
^{+^{(x,t)}}_{r''}(\Delta _r^+)$ takes the boundary value $\chi _{\Delta
  _r^+}$ on $\partial _pD_{r'}^+\cap \partial _pD$, then by the maximum
principle we have $u_{n,r'}^+(x,t)\leq \vartheta_+^{(x,t)}(\Delta _r^+)$
for $(x,t)\in D_{r'}^+$. Similarly, $u_{n,r'}^-(x,t)\leq
\vartheta_-^{(x,t)}(\Delta _r^-)$ for $(x,t)\in D_{r'}^-$. Therefore, for
$(x,t)\in D_{r'}^+\cap D_{r'}^-$ and $0<r'<r$ sufficiently small we
have  
\begin{equation}\label{eq:twofs}
u_{n,r'}^+(x,t)+u_{n,r'}^-(x,t)\leq \vartheta_+^{(x,t)}(\Delta _r^+)+\vartheta_-^{(x,t)}(\Delta _r^-).
\end{equation}
Let $r'\searrow 0$, then $D_{r'}^+\cap D_{r'}^-\nearrow D$. By the comparison principle there is a nonnegative function $\tilde{u}_n$ in $\Psi_1$ and caloric in $D$ such that
\begin{equation}\label{eq:tildeu}
u^+_{n,r'}(x,t)+u^-_{n,r'}(x,t)\nearrow \tilde{u}_n (x,t)\text{ as } r'\searrow 0,\quad (x,t)\in D.
\end{equation}
By (i) just shown above and \eqref{eq:twofs},
\begin{align*}
& \sup _{\partial_pD_{r'}^+\cap D}u_{n,r'}^+(x,t)+\sup _{\partial_pD_{r'}^-\cap D}u_{n,r'}^-(x,t)\\
& \leq \sup _{\partial _pD_{r'}^+\cap D}\vartheta_+^{(x,t)}(\Delta _r^+)+\sup _{\partial_pD_{r'}^-\cap D}\vartheta_-^{(x,t)}(\Delta _r^-)\rightarrow 0\quad \text{ as } r'\rightarrow 0,
\end{align*}
hence it is not hard to see that
$\tilde{u}_n$ takes the boundary value $g_n$ continuously on $\partial _pD$. Hence by the maximum principle $\tilde{u}_n=u_n$ in $D$. This combined
with \eqref{eq:twofs} and \eqref{eq:tildeu} gives 
\begin{equation}\label{eq:un}
u_n(x,t)\leq \vartheta_+^{(x,t)}(\Delta _r^+)+\vartheta_-^{(x,t)}(\Delta _r^-).
\end{equation}
Letting $n\rightarrow \infty$ in \eqref{eq:un}, we obtain
$$\omega ^{(x,t)}(\Delta_r)\leq \vartheta_+^{(x,t)}(\Delta _r^+)+\vartheta_-^{(x,t)}(\Delta _r^+).$$

By taking the approximation $g_n\searrow \chi _{\Delta _r}$, $0\leq
g_n\leq 2$ and $\supp {g_n}\subset \cN_{2r}(E_f)\cap \partial_pD $ we obtain the reverse
inequality and hence the equality.

\step{Proof of (iii)}: We only show it for $\vartheta_+$ and assume additionally $\Delta^\pm_{r'}\not=\emptyset$.

First for $0<r''<r'<r_0$, by Lemma~1.1 in \cite{Kemper} there exists $C=C(n)\geq 0$ such that 
\begin{equation*}
\omega_{\Psi_{2r'}(y,s)\cap D_+}^{\overline{A}_{r'}^+(y,s)}(\Delta_{r'}^+)\geq C.
\end{equation*}
Applying the comparison principle in $\Psi_{2r'}(y,s)\cap D_+$ we have
\begin{equation}\label{eq:theta-lemma2}
\vartheta_+^{\overline{A}_{r'}^+(y,s)}(\Delta_{r'}^+)\geq C.
\end{equation}  

Next for $0<r''<r'<r_0$, applying the same induction arguments as in \\
Lemma~\ref{lem:caloric2} we have
\begin{equation}\label{eq:cal}
\omega_{r''}^{+^{(x,t)}}(\Delta _{r'}^+)\leq C \omega_{r''}^{+^{\overline{A}^+_r(y,s)}}(\Delta _{r'}^+),\quad \text{for }(x,t)\in D^+_{r''}\setminus (\Psi_r(y,s))_+,
\end{equation}
where $C=C(n,L)$ is
independent of $r'$ and $r''$. The reason that $C$ is uniform in $r''$
is as follows. By the maximum principle it is enough to show
\eqref{eq:cal} for $(x,t)\in \partial (\Psi_r(y,s))_+\cap D^+_{r''}$,
which is contained in $D_+$. Hence the same iteration procedure as in
Lemma~\ref{lem:caloric2} but only on the $D_+$ side gives
\eqref{eq:cal}, and the proof is uniform in $r''$. Therefore, letting
$r''\rightarrow 0$ in \eqref{eq:cal}, we obtain 
$$\vartheta_+^{(x,t)}(\Delta _{r'}^+)\leq C \vartheta_+^{\overline{A}^+_r(y,s)}(\Delta _{r'}^+).$$
Applying Lemma~\ref{lem:caloric1} and the maximum principle, we deduce
(iii).

\step{Proof of (iv)}: Applying (iii), (ii), Harnack inequality and
Lemma~\ref{lem:caloric1} we have that for given $(y,s)\in E_f\setminus
G_f$ and $0<r<r_0$, 
\begin{align*}
\vartheta^{(X,T)}_-(\Delta _r^-)& \leq C\vartheta_-^{\overline{A}_{r_0}^-(y,s)}(\Delta_r^-)\leq C\omega^{\overline{A}_{r_0}^-(y,s)}(\Delta_r)\\
&\leq C \omega^{\overline{A}_{2r_0}^+(y,s)}(\Delta_r)\leq C\vartheta^{\overline{A}_{2r_0}^+(y,s)}_+(\Delta_r^+)\\
&\leq C\vartheta_+^{(X,T)}(\Delta_r^+),
\end{align*}
for $C=C(n,L,r_0)$. The second last inequality holds because 
\begin{equation}\label{eq:theta-comparison}
\vartheta^{\overline{A}_{2r_0}^+(y,s)}_+(\Delta_r^+)\geq \vartheta^{\overline{A}_{2r_0}^+(y,s)}_-(\Delta_r^-),
\end{equation}
which follows from the $x_n$ symmetry of $D$ and the comparison principle. \eqref{eq:theta-comparison} together with (ii) just shown above yield the result. 

\end{proof}

Now we use $\vartheta_+$ and $\vartheta_-$ to construct two linear
independent kernel functions at $(y,s)\in E_f\setminus G_f$.  

\begin{theorem}\label{thm:twokernel}
Given $(y,s)\in E_f\setminus G_f$, there exist at least two linearly
independent kernel functions at $(y,s)$. 
\end{theorem}

\begin{proof}
Given $(y,s)\in E_f\setminus G_f$, let $r_0$ be as in
\eqref{eq:r0}. For $m>1/r_0$ we consider the sequence 
\begin{equation}\label{eq:kernelplus}
v_m^+(x,t)=\frac{\vartheta_+^{(x,t)}(\Delta
  _{\frac{1}{m}}^+(y,s))}{\vartheta_+^{(X,T)}(\Delta
  _{\frac{1}{m}}^+(y,s))},\quad (x,t)\in D. 
\end{equation}
By Proposition~\ref{prop:limitfunc}(iii) and the same arguments as in
Section~\ref{sec:exist-kern-funct}, we have, up to a subsequence, that
$v_m(x,t)$ converges to a kernel function at $(y,s)$ normalized at
$(X,T)$. We denote it by $K^+(x,t;y,s)$. 

If we consider instead 
\begin{equation}\label{eq:kernelminus}
v_m^-(x,t)=\frac{\vartheta_-^{(x,t)}(\Delta _{\frac{1}{m}}^-(y,s))}{\vartheta_-^{(X,T)}(\Delta _{\frac{1}{m}}^-(y,s))},\quad (x,t)\in D,
\end{equation}
we will obtain another kernel function at $(y,s)$, which we will
denote $K^-(x,t;y,s)$. 

We now show that for $(y,s)$ fixed, $K^+(\cdot,\cdot;y,s)$ and $
K^-(\cdot,\cdot;y,s)$ are linearly independent. In fact, by
Proposition~\ref{prop:limitfunc}(i), \eqref{eq:kernelplus} and
\eqref{eq:kernelminus} we have $K^+(x,t;y,s)\rightarrow 0$ as
$(x,t)\rightarrow(y,s)$ from $D_-$ and $K^-(x,t;y,s)\rightarrow 0$ as
$(x,t)\rightarrow(y,s)$ from $D_+$. If
$K^+(\cdot,\cdot;y,s)=K^-(\cdot, \cdot;y,s)$, then we also have
$K^+(x,t;y,s)\rightarrow 0$ as $(x,t)\rightarrow (y,s)$ from $D_+$,
which will mean that $K^+(x,t; y,s)$ is a caloric function continuously
vanishing on the whole $\partial _pD$. By the maximum principle $K^+$
will vanish in the entire $D$, which contradicts the 
normalization condition $K^+(X,T;y,s)=1$. Moreover, since $K^+(X,T;y,s)=K^-(X,T;y,s)=1$,
it is impossible that $K^+(\cdot,\cdot;y,s)=\lambda
K^-(\cdot,\cdot;y,s)$ for a constant $\lambda \neq 1$. Hence $K^+$ and
$K^-$ are linearly independent. 
\end{proof}

\begin{remark}\label{rem:Martin}
The non-uniqueness of the kernel functions at $(y,s)$ shows that the parabolic Martin boundary of $D$ is not homeomorphic to Euclidean parabolic boundary $\partial _pD$. 
\end{remark}

Next we show $K^+$ and $K^-$ in fact span the space of all the kernel
functions at $(y,s)$. We use an argument similar to the one in \cite{Kemper}.

\begin{lemma}\label{lem:kernel3}
Let $(y,s)\in E_f\setminus G_f$. There exists a positive constant $C=C(n,L,r_0)$ such that if $u$ is a kernel function at $(y,s)$ in $D$, we have either
\begin{equation}\label{eq:uniquekernel}
u\geq CK^+,
\end{equation}
or \begin{equation}\label{eq:uniquekernel2}
u\geq CK^-.
\end{equation}
Here $K^+$, $K^-$ are the kernel functions at $(y,s)$ constructed from \eqref{eq:kernelplus} and \eqref{eq:kernelminus}.
\end{lemma}

\begin{proof}
For $0<r<r_0$ we consider $u^\pm_r:D_r^\pm\rightarrow \R$, where $u^\pm_r(x,t)=u(F_r^\pm(x,t))$. $u^\pm _r$ are caloric in $D^{\pm}_r$ and continuous up to the boundary. Then for $(x,t)\in D^{\pm}_r$,
\begin{align*}
u^{\pm}_r(x,t)&= \int_{\partial _pD^{\pm}_r}u_r^{\pm}(z,h)d\omega
_r^{{\pm}^{(x,t)}}(z,h)\geq \int_{\Delta _r^\pm(y,s)}u_r^{\pm}(z,h)d\omega
_r^{{\pm}^{(x,t)}}(z,h)\\ 
&\geq  \inf_{(z,h)\in \Delta _r^\pm(y,s)}u_r^{\pm}(z,h)\omega
^{{\pm}^{(x,t)}}_r(\Delta _r^\pm(y,s)). 
\end{align*}
Note that the parabolic distance between $F^{\pm}_r(\Delta _r^\pm(y,s))$
and $\partial _pD$ is equivalent to $r$ and the time lag between it
and $\overline{A}^{\pm}_r(y,s)$ is equivalent to $r^2$, hence by the
Harnack inequality there exists $C=C(n,L)$ such that 
$$
\inf_{(z,h)\in \Delta _r^\pm(y,s)}u^{\pm}_r(z,h)\geq
Cu(\overline{A}^{\pm}_r(y,s)).
$$
Hence,
\begin{equation}\label{eq:kernel3}
u^{\pm}_r(x,t)\geq C u(\overline{A}^{\pm}_r(y,s))\omega
^{{\pm}^{(x,t)}}_r(\Delta _r^\pm(y,s)),\quad \text{for}\ (x,t)\in
D^{\pm}_r. 
\end{equation}
On the other hand, $u$ is a kernel function at $(y,s)$ and vanishes on $\partial _pD\setminus \Delta _{r/4}(y,s)$ for any
$0<r<1$. Applying Theorem~\ref{thm:carleson2} we obtain 
\begin{multline}\label{eq:kernel4}
u(x,t)\leq C \max\{u(\overline{A}_{r/2}^+(y,s)),
u(\overline{A}_{r/2}^-(y,s))\}\omega ^{(x,t)}(\Delta
_r(y,s)),\\\text{for }(x,t)\in D\setminus \Psi_{r/2}(y,s).
\end{multline}

\case{Case 1.}
$u(\overline{A}^+_{r/2}(y,s))\geq u(\overline{A}^-_{r/2}(y,s))$ in \eqref{eq:kernel4}.

By Proposition~\ref{prop:limitfunc}(ii) and the Harnack inequality,
\begin{equation*}
u(x,t)\leq C u(\overline{A}_r^+(y,s))(\vartheta_+^{(x,t)}(\Delta _r^+)+\vartheta_-^{(x,t)}(\Delta _r^-)),\quad (x,t)\in D\setminus \Psi_{r/2}(y,s)
\end{equation*}
In particular,
\begin{equation}\label{eq:kernel5}
1=u(X,T)\leq C u(\overline{A}_r^+(y,s))(\vartheta_+^{(X,T)}(\Delta^+ _r)+\vartheta_-^{(X,T)}(\Delta^- _r)).
\end{equation}
Now \eqref{eq:kernel3} for $u^+_r$, \eqref{eq:kernel5} and Proposition~\ref{prop:limitfunc}(iv) yield the existence of $C_1=C_1(n,L, r_0)$ such that for any $0<r<r_0$,
\begin{equation}\label{eq:kernelcase1}
u^+_r(x,t)\geq C\frac{\omega_r^{+^{(x,t)}}(\Delta _r^+)}{\vartheta_+^{(X,T)}(\Delta _r^+)+\vartheta_-^{(X,T)}(\Delta _r^-)}\geq C_1\frac{\omega_r^{+^{(x,t)}}(\Delta _r^+)}{\vartheta_+^{(X,T)}(\Delta _r^+)}, \quad (x,t)\in D_r^+.
\end{equation}
Since by the maximum principle in $D_r^+$
\begin{equation}\label{eq:kernel10}
\omega ^{+^{(x,t)}}_r(\Delta _r^+)\geq \vartheta_+^{(x,t)}(\Delta _r^+)-\sup_{(z,h)\in \partial _pD^+_r\cap D}\vartheta_+^{(z,h)}(\Delta _r^+),
\end{equation} 
then \eqref{eq:kernelcase1} can be written as
\begin{equation}\label{eq:kernel100}
u^+_r(x,t)\geq C_1\left(\frac{\vartheta_+^{(x,t)}(\Delta _r^+)}{\vartheta_+^{(X,T)}(\Delta _r^+)}-\sup_{(z,h)\in \partial _pD^+_r\cap D}\frac{\vartheta_+^{(z,h)}(\Delta _r^+)}{\vartheta_+^{(X,T)}(\Delta _r^+)}\right),\quad (x,t)\in D_r^+.
\end{equation}
By Proposition~\ref{prop:limitfunc}(iii) and the Harnack inequality, there exists $C_2=C_2(n,L,r_0)$ such that for $(z,h)\in \partial _pD^+_r\cap D$,
\begin{equation}\label{eq:kernel11}
\frac{\vartheta^{(z,h)}_+(\Delta _r^+)}{\vartheta^{(X,T)}_+(\Delta _r^+)}\leq C\frac{\vartheta_+^{\overline{A}_{r_0}^+}(\Delta _r^+)}{\vartheta_+^{(X,T)}(\Delta _r^+)} \cdot \vartheta^{(z,h)}_+(\Delta _{r_0}^+)\leq C_2 \vartheta^{(z,h)}_+(\Delta _{r_0}^+),
\end{equation}
Hence \eqref{eq:kernel100} and \eqref{eq:kernel11} imply
\begin{equation*}
u^+_r(x,t)\geq C_1\left(\frac{\vartheta_+^{(x,t)}(\Delta _r^+)}{\vartheta^{(X,T)}_+(\Delta _r^+)}-C_2\sup_{(z,h)\in \partial _pD_r^+\cap D}\vartheta^{(z,h)}_+(\Delta _{r_0}^+)\right), \quad (x,t)\in D^+_r.
\end{equation*}

\case{Case 2.}
$u(\overline{A}^+_{r/2}(y,s))\leq u(\overline{A}^-_{r/2}(y,s))$ in \eqref{eq:kernel4}. Similarly,
\begin{equation*}
u^-_r(x,t)\geq C_1\left(\frac{\vartheta_-^{(x,t)}(\Delta _r^-)}{\vartheta^{(X,T)}_-(\Delta _r^-)}-C_2\sup_{(z,h)\in \partial _pD_r^-\cap D}\vartheta^{(z,h)}_-(\Delta _{r_0}^-)\right), \quad (x,t)\in D^-_r.
\end{equation*}
Note that as $r\searrow 0$, $D_r^\pm \nearrow D$, and $u_r^\pm
\rightarrow u$. Let $r_j\rightarrow 0$ be such that either
Case 1 applies for all $r_j$ or Case 2 applies. Hence, over a subsequence, it follows  by 
Proposition~\ref{prop:limitfunc}(i) and \eqref{eq:kernelplus} that either 
\begin{align*}
u(x,t)&\geq C_1\lim_{r_j\rightarrow 0} \left(\frac{\vartheta_+^{(x,t)}(\Delta _{r_j}^+)}{\vartheta^{(X,T)}_+(\Delta _{r_j}^+)}-C_2\sup_{(z,h)\in \partial _pD_{r_j}^+\cap D}\vartheta^{(z,h)}_+(\Delta _{r_0}^+)\right)\\
&=C_1K^+(x,t), \quad \text{for all }(x,t)\in D,
\end{align*}
or
\begin{equation*}
u(x,t)\geq C_1K^-(x,t), \quad \text{for all }(x,t)\in D.
\end{equation*}
\end{proof}

The next theorem says that $K^+(\cdot,\cdot;y,s)$ and $K^-(\cdot,\cdot;y,s)$ span the space of kernel functions at $(y,s)$.

\begin{theorem}\label{thm:kernel8}
If $u$ is a kernel function at $(y,s)\in E_f\setminus G_f$ normalized
at $(X,T)$, then there exists a constant $\lambda \in [0,1]$ which may
depend on $(y,s)$, such that $u(\cdot,\cdot)=\lambda
K^+(\cdot,\cdot;y,s)+(1-\lambda)K^-(\cdot,\cdot;y,s)$ in $D$, where
$K^+$ and $K^-$ are kernel function obtained from
\eqref{eq:kernelplus} and \eqref{eq:kernelminus}. 
\end{theorem}

\begin{proof}
By Lemma~\ref{lem:kernel3} if $u$ is a kernel function at $(y,s)$, then either (i) $u\geq CK^+$ or (ii) $u\geq CK^-$ with $C=C(r_0,n,L)$.

If (i) takes place, let
$$\lambda =\sup\{C: u(x,t)\geq CK^+(x,t), \forall (x,t)\in D\},$$
then we must have $\lambda \leq 1$, because $u(X,T)=K^+(X,T)=1$. If
$\lambda =1$, then $u(x,t)=K^+(x,t)$ for all $(x,t)\in D$ by the strong maximum principle and we are
done. If $\lambda <1$, consider 
$$u_1(x,t):=\frac{u(x,t)-\lambda K^+(x,t)}{1-\lambda},$$
which is another kernel function at $(y,s)$ satisfying either (i) or (ii). If (i) holds for $u_1$ for some $C>0$, then $u(x,t)\geq (C(1-\lambda )+\lambda )K^+(x,t)$, with $C(1-\lambda)+\lambda >\lambda$ which contradicts to the supreme of $\lambda$. Hence (ii) must be true for $u_1$. Let $$\tilde{\lambda}=\sup\{C: u_1(x,t)\geq CK^-(x,t),\forall (x,t)\in D\}.$$
The same reason as above gives $\tilde{\lambda}\leq 1$. We claim
$\tilde{\lambda} =1$. 

Proof of the claim: If not, then $\tilde{\lambda} <1$. We get
$$u_2(x,t):=\frac{u_1(x,t)-\tilde{\lambda} K^-(x,t)}{1-\tilde{\lambda}}$$
is again a kernel function at $(y,s)$. If $u_2$ satisfies (i) for some $C>0$, then
$$u_1(x,t)\geq u_1(x,t)-\tilde{\lambda} K^-(x,t)\geq C(1-\tilde{\lambda} )K^+(x,t),$$
which implies
$$u(x,t)\geq (\lambda +C(1-\tilde{\lambda}))K^+(x,t)$$
is again a contradiction to the supreme of $\lambda$. Hence $u_2$ has to satisfy (ii) for some $C>0$, then we have
$$u_2(x,t)\geq (C(1-\tilde{\lambda})+\tilde{\lambda})K^-(x,t),$$
but this contradicts to the supreme of $\tilde{\lambda}$. Hence we proved the claim.

The fact that $\tilde{\lambda}=1$ implies that $u_1(x,t)=K^-(x,t)$ in $D$ by the strong maximum principle. Hence if (i) applies to $u$ we have $u(x,t)=\lambda K^+(x,t)+(1-\lambda)K^-(x,t)$ with $\lambda \in (0,1]$. If (ii) applies to $u$ we get the equality with $\lambda \in [0,1)$.
\end{proof}

\subsection{Radon-Nikodym derivative as a kernel function}
We first show that the kernel function at $(y,s)\in G_f$ or $(y,s)\in \partial _pD\setminus
E_f$ is unique. The proof
for the uniqueness is similar as Lemma~1.6 and Theorem~1.7 in
\cite{Kemper}. More precisely, we   
will need the following direction shift operator $F^0_r$:   
\begin{align}
F^0_r(x,t)&=(x'', x_{n-1}+4nLr, x_n, t+8r^2), \quad 0<r<1/4\label{eq:shift}\\
D^0_r&=\{(x,t)\in D: F^0_r(x,t)\in D\}.\notag
\end{align}
Let ${\omega}^0_r $ denote the caloric measure for $D^0_r$.
Note that $D^0_r$ is also a cylindrical domain with a thin Lipschitz complement.

\begin{theorem}\label{thm:kernel1}
For all $(y,s)\in \partial _pD$ the limit of \eqref{eq:kernel} exists. If we denote the limit by $K_0(\cdot,\cdot; y,s)$, i.e.
\begin{equation*}
K_0(x,t;y,s)=\lim_{n\rightarrow \infty}\omega^{(x,t)}(\Delta
_{\frac{1}{n}}(y,s))/\omega ^{(X,T)}(\Delta _{\frac{1}{n}}(y,s)). 
\end{equation*}
then 
\mbox{}
\begin{itemize}
\item[(i)] For $(y,s)\in G_f$ or $(y,s)\in \partial_pD\setminus E_f$, $K_0$ is the unique kernel function at $(y,s)$.

\item[(ii)] If $(y,s)\in E_f\setminus G_f$, then $K_0$ is a kernel function at $(y,s)$ and
\begin{equation}\label{eq:halfkernel}
K_0(x,t;y,s)=\frac{1}{2}K^+(x,t;y,s)+\frac{1}{2}K^-(x,t;y,s),
\end{equation}
where $K^+$ and $K^-$ are kernel functions at $(y,s)$ given by the limit of \eqref{eq:kernelplus} and \eqref{eq:kernelminus}.
\end{itemize}
\end{theorem}
\begin{proof}
For $(y,s)\in G_f$ and $r$ small enough, we denote
$\overline{\overline{A}}_r(y,s)=(y'',y_{n-1}+4nrL, 0,s+4r^2)$, which is on
$\{x_{n}=0\}$ and have a time-lag $2r^2$ above
$\overline{A}_r^{\pm}$. Then by the Harnack inequality,  
$$\omega ^{\overline{A}_r^{\pm}(y,s)}(\Delta _{r'}(y,s))\leq C(n,L)\omega ^{\overline{\overline{A}}_r(y,s)}(\Delta_{r'}(y,s)), \quad \forall 0<r'<r.$$ 
Then one can proceed as in Lemma~1.6 of \cite{Kemper} by using $F^0_r$, $D^0_r$, ${\omega}^0$ to show that any kernel function (at $(y,s)$) $u$ satisfies $u\geq CK_0$ for some $C>0$.  Then the uniqueness follows from Theorem~1.7, Remark~1.8 of \cite{Kemper}. 

For $(y,s)\in \partial _pD\setminus E_f$, for $r$ sufficiently small one has either $\Psi_r(y,s)\cap D\subset D_+$ or $\Psi_r(y,s)\cap D\subset D_-$. In either case one can proceed as in Lemma~1.6, Theorem~1.7 and Remark~1.8 of \cite{Kemper}. 

For $(y,s)\in E_f\setminus G_f$, by Theorem~\ref{thm:kernel8}, $K_0(x,t;y,s)=\lambda K^+(x,t; y,s)+(1-\lambda)K^-(x,t;y,s)$ for some $\lambda\in [0,1]$. By Proposition~\ref{prop:limitfunc}(ii), the symmetry of the domain about $x_{n-1}$ and the definition of $K^\pm$, one has $\lambda =1/2$.
\end{proof}

\begin{remark}\label{rem:representation}
From Theorem~\ref{thm:kernel1} we can conclude that the Radon-Nikodym derivative $d\omega ^{(x,t)}/d\omega^{(X,T)}$ exists at every $(y,s)\in \partial _pD$ and it is the kernel function $K_0(x,t;y,s)$ with respect to $(X,T)$. 

\end{remark}

The following corollary is an easy consequence of Theorems
\ref{thm:kernel8} and \ref{thm:kernel1}. 

\begin{corollary}\label{cor:cont}
For fixed $(x,t)\in D$, the function $(y,s)\mapsto K_0(x,t;y,s)$ is continuous on $\partial _pD$, where $K_0$ is given by the limit of \eqref{eq:kernel}.
\end{corollary}

\begin{proof}
Given $(y,s)\in \partial _pD$, let $(y_m,s_m)\in \partial _pD$ with $(y_m,s_m)\rightarrow (y,s)$ as $m\rightarrow \infty$.

If $(y,s)\in G_f$ or $\partial _pD\setminus E_f$, continuity follows
from the uniqueness of the kernel function.

If $(y,s)\in E_f\setminus G_f$, by Theorem~\ref{thm:kernel1}(ii) for each $m$ we have
\begin{equation}\label{eq:kernel9}
K_0(x,t;y_m,s_m)=\frac{1}{2}K^+(x,t;y_m,s_m)+\frac{1}{2}K^-(x,t;y_m,s_m).
\end{equation}
Given $\epsilon>0$, $K^+(\cdot,\cdot;y_m,s_m)$ is uniformly bounded
and equicontinuous on $D\setminus \Psi_{\epsilon}(y,s)$ for $m$ large
enough, hence by a similar argument as in
Section~\ref{sec:exist-kern-funct}, up to a subsequence,
$K^+(\cdot,\cdot;y_m,s_m)\rightarrow v^+(\cdot,\cdot;y,s)$ uniformly
on compact subsets, where $v^+(\cdot, \cdot;y,s)$ is some kernel function
at $(y,s)$. Moreover, by Theorem~\ref{thm:kernel8} we have 
\begin{equation}\label{eq:kernelv}
v^+(\cdot,\cdot;y,s)=\lambda K^+(\cdot,\cdot;y,s)+(1-\lambda)K^-(\cdot,\cdot;y,s),\quad \text{ for some }\lambda \in [0,1].
\end{equation}

By Proposition~\ref{prop:limitfunc}(i),
$$\sup_{(x,t)\in\partial _pD^+_r\cap D}K^+(x,t;y_m,s_m)\rightarrow 0,\quad r\rightarrow 0$$
which is uniform in $m$ from the proof of the proposition. Hence after $m\rightarrow \infty$, $v^+$ satisfies
$$\sup_{(x,t)\in \partial _pD^+_r\cap D}v^+(x,t)\rightarrow 0, \quad r\rightarrow 0,$$
which combined with
$$K^-(x,t;y,s)\not\rightarrow 0, \text{ as } (x,t)\rightarrow (y,s),\quad \text{for }(x,t)\in D_-$$
gives $\lambda =1$ in \eqref{eq:kernelv}. 

Similarly, up to a subsequence $K^-(x,t;y_m,s_m)\rightarrow K^-(x,t;y,s).$

Thus along a subsequence $K(\cdot,\cdot;y_m,s_m)\rightarrow K_0(\cdot,\cdot;y,s)$ by \eqref{eq:halfkernel}. Since this holds for all the converging subsequences, then $K_0(x,t;y,s)$ is continuous on $\partial_pD$ for fixed $(x,t)$.
\end{proof}

By using Corollary~\ref{cor:cont}, Remark~\ref{rem:representation} and Theorem~\ref{thm:carleson2} we can prove some uniform behavior of $K_0$ on $\partial _pD$ as in Lemmas~2.2 and 2.3 of \cite{Kemper}. We state the results in the following two lemmas and omit the proof.

\begin{lemma}\label{lem:kernel}
Let $(y,s)\in \partial _pD$. Then for $0<r<1/4$,
\begin{equation*}
\sup_{(y',s')\in \partial _pD\setminus \Delta _r(y,s)}K_0(x,t;y',s')\rightarrow 0,\ \text{as}\ (x,t)\rightarrow (y,s)\ \text{in}\ D.
\end{equation*}
\end{lemma}

The following lemma says that if $D'$ is a domain obtained by a
perturbation of a portion of $\partial_pD$ where $\omega ^{(x,t)}$
vanishes, then the caloric measure $\omega _{D'}$ is equivalent to
$\omega _D$ on the common boundary of $D' $ and $D$. 
We recall here $\omega^0_r$ is the caloric measure with respect to the domain 
$D^0_r$ defined in \eqref{eq:shift} and $\omega_r^\pm$ is the caloric measure with respect to
$D^\pm_r$ defined in \eqref{eq:shift00}.

\begin{lemma}\label{lem:kernel2}\mbox{}
\begin{enumerate}
\item[(i)] Let $r\in (0, 1/4)$ and $(y,s)\in G_f\cup(\partial
  _pD\setminus E_f)$ with $s>-1+4r^2$. Then
there exists $\rho_0=\rho_0(n,L)>0$, $C=C(n,L)>0$ such that for
$0<\rho<\rho_0$ we have 
\begin{equation}\label{eq:kernel2}  
{\omega }^{0^{(X',T')}}_\rho(\Delta _r(y,s))\geq C \omega
^{(X',T')}(\Delta _r(y,s)), \quad (X',T')\in \Psi_{1/4}(X,T),
\end{equation}
provided also $r<|y_{n}|$ for $(y,s)\in \partial_pD\setminus E_f$.
\item[(ii)] Let $(y,s)\in (\cN_r(E_f)\cap \partial_pD)\setminus G_f$. Then
there exists $\delta_0=\delta_0(n,L)>0$, such that for $0<r'<\delta_0$ we have
\begin{equation}\label{eq:kernel22}
\omega ^{+^{(X',T')}}_{r'} (\Delta _r^+(y,s))+\omega ^{-^{(X',T')}}_{r'} (\Delta _r^-(y,s))\geq \frac{1}{2} \omega ^{(X',T')}(\Delta _r(y,s))
\end{equation}
for $(X',T')\in \Psi_{1/4}(X,T)$ and $0<r<r_0$, where $r_0$ is the constant defined in \eqref{eq:r0}.
\end{enumerate}
\end{lemma}

\begin{proof}
To show \eqref{eq:kernel22} we first argue similarly as in \cite{Kemper} to show there exists $\delta_0=\delta_0(n,L)>0$ such that
for any $0<r'<\delta_0$ 
\begin{equation}\label{eq:theta-omega}
\omega_{r'}^{\pm^{(X',T')}}(\Delta_r^\pm(y,s))\geq \frac{1}{2}\vartheta_\pm^{(X',T')}(\Delta_r^\pm(y,s))
\end{equation}
for each $\Delta_r^\pm(y,s)$ with $0<r<r_0$. Then using Proposition \ref{prop:limitfunc}(ii) we get the conclusion. 
\end{proof}

\section{Backward Boundary Harnack Principle}
\label{sec:backw-bound-harn}

In this section, we follow the lines of \cite{Garofalo3} to build up a
backward Harnack inequality for nonnegative caloric functions in
$D$. To prove this kind of inequalities, we have to ask the
nonnegative caloric functions to vanish on the \emph{lateral boundary}
$$
S:=\partial_pD\cap\{s>-1\},
$$
or at least a portion of it. This will allow to control the time-lag
issue in the parabolic Harnack inequality.

Some of the proofs in this section follow the lines of the
corresponding proofs in \cite{Garofalo3}. For that reason, we will
omit the parts that don't require modifications or additional arguments.

For $(x,t)$ and $(y,s)\in D$, denote by $G(x,t;y,s)$ the Green's function
for the heat equation in the domain $D$. Since $D$ is a regular
domain, Green's function can be written in the form
$$
G(x,t;y,s)=\Gamma(x,t;y,s)-V(x,t;y,s),
$$
where $\Gamma(\cdot,\cdot;y,s)$ is the fundamental solution of the
heat equation with pole at $(y,s)$ and $V(\cdot,\cdot;y,s)$ is a
caloric function in $D$ that equals  $\Gamma(\cdot,\cdot;y,s)$ on
$\partial_p D$. We note that by the maximum principle we have
$G(x,t;y,s)=0$ whenever $(x,t)\in D$ with $t\leq s$.

In this section, similarly to Section~\ref{sec:kernel-functions}, we will work under
Convention~\ref{convention}. In particular, in Green's function we
will allow pole $(y,s)$ to be in $\tilde D$ with $s\geq 1$. But in
that case we simply have $G(x,t;y,s)=0$ for all $(x,t)\in D$.

\begin{lemma}\label{lem:green}
Let $0<r<1/4$ and $(y,s)\in S$ with $s\geq -1+8r^2$. Then
there exists a constant $C=C(n,L)>0$ such that for $(x,t)\in D\cap\{t\geq s+4r^2\}$, we have
\begin{multline}\label{eq:green}
C^{-1}r^n\max\{G(x,t;\overline{A}^\pm_r(y,s))\}\leq \omega
^{(x,t)}(\Delta _r(y,s))\\\leq
Cr^n\max\{G(x,t;\underline{A}^\pm_r(y,s))\},\quad\text{if }(y,s)\in \cN_r(E_f),
\end{multline}
\begin{multline}\label{eq:green-2}
C^{-1}r^n G(x,t;\overline{A}_r(y,s))\leq \omega ^{(x,t)}(\Delta
_r(y,s))\\\leq Cr^n G(x,t;\underline{A}_r(y,s)),\quad\text{if }(y,s)\not\in \cN_r(E_f).
\end{multline}
\end{lemma}

\begin{proof}
The proof uses Lemma~\ref{lem:caloric1} and Theorem~\ref{thm:carleson}
and is similar to that of Lemma~1 in \cite{Garofalo3}. 
\end{proof}

\begin{theorem}[Interior backward Harnack inequality]
\label{thm:interior}
Let $u$ be a positive caloric function in $D$ vanishing continuously
on $S$. Then for any compact $K\Subset D$ there exists a constant $C=C(n,L, \dist(K,\partial_pD))$ such that
\begin{equation*}
\max_{K}{u}\leq C\min_{K}{u}
\end{equation*}
\end{theorem}

\begin{proof}
 The proof is similar to that of Theorem~1 in \cite{Garofalo3} and
 uses Theorem~\ref{thm:carleson} and the Harnack inequality.
\end{proof}

\begin{theorem}[Local comparison theorem]
\label{thm:local}
Let $0<r<1/4$ and $(y,s)\in S$ with $s\geq -1+18r^2$, and $u,v$ be two positive caloric functions in
$\Psi_{3r}(y,s)\cap D$ vanishing continuously on
$\Delta_{3r}(y,s)$. Then there exists $C=C(n,L)>0$ such that for $(x,t)\in \Psi_{r/8}(y,s)\cap D$
we have: 
\begin{alignat}{2}\label{eq:local}
\frac{u(x,t)}{v(x,t)}&\leq C\frac{\max\{u(\overline{A}_r^+(y,s)),
  u(\overline{A}_r^-(y,s))\}}{\min\{v(\underline{A}_r^+(y,s)),
  v(\underline{A}_r^-(y,s))\}},&&\quad\text{if }(y,s)\in\cN_r(E_f)
\intertext{and}
\label{eq:local-2}
\frac{u(x,t)}{v(x,t)}&\leq C\frac{u(\overline{A}_r(y,s))}{v(\underline{A}_r(y,s))},&&\quad\text{if }(y,s)\not\in\cN(E_f).
\end{alignat}
\end{theorem}

\begin{proof}
The proof is similar to that of Theorem~3 in \cite{Garofalo3}. 
First, note that if $\Psi_{r/8}(y,s)\cap E_f=\emptyset$, we can
consider restriction of $u$ and $v$ to $D_+$ or $D_-$ (which are Lipschitz cylinders)
and apply the arguments from \cite{Garofalo3} directly there. Thus, we may
assume that $\Psi_{r/8}(y,s)\cap E_f\not=\emptyset$. If we now argue
as in the proof of the localization property
(Lemma~\ref{lem:loc-property}) by replacing $(y,s)$ and $r$ with
$(\tilde y, \tilde s)\in \Psi_{(3/8)r}(y,s)\cap E_f$ we may further assume that $(y,s)\in E_f$, and that $\Psi_r(y,s)\cap D$ falls either into
category (2) or (3) in the localization property. For definiteness, we
will assume category (3). To account for the possible
change in $(y,s)$ we then change the hypothesis to $u=0$ on
$\Delta_{2r}(y,s)$ and prove \eqref{eq:local} for $(x,t)\in\Psi_{r/2}(y,s)\cap D$.

With the above simplification in mind, we proceed as in the proof of
Theorem~3 in \cite{Garofalo3}. By using Lemma~\ref{lem:green} and
Theorem~\ref{thm:carleson2} we first show
\begin{equation}\label{eq:alphabeta}
\omega_r^{(x,t)}(\alpha_r)\leq C\omega_r^{(x,t)}(\beta_r),\quad (x,t)\in \Psi_{r/2}(y,s)\cap D
\end{equation}
where $\alpha_r=\partial_p (\Psi_r(y,s)\cap D)\setminus S$,
$\beta_r=\partial_p (\Psi_r(y,s)\cap D)\setminus \cN_{\mu r}(S)$ with
a small fixed $\mu\in (0,1)$, and where $\omega _r$
denotes the caloric measure with respect to $\Psi_r(y,s)\cap D$.
Then by Theorem~\ref{thm:carleson}, Harnack inequality, and the
maximum principle we obtain
\begin{align*}
u(x,t)&\leq C\max\{u(\overline{A}_r^+(y,s)),
  u(\overline{A}_r^-(y,s))\} \omega_r^{(x,t)}(\alpha_r)\\
v(x,t)&\geq C\min\{v(\underline{A}_r^+(y,s)),
  v(\underline{A}_r^-(y,s))\}\omega_r^{(x,t)}(\beta_r),
\end{align*}
which combined with \eqref{eq:alphabeta} completes the proof. 
\end{proof}

\begin{theorem}[Global comparison theorem]
\label{thm:global}
Let $u,v$ be two positive caloric functions in $D$, vanishing
continuously on $S$, and let $(x_0,t_0)$ be a fixed point in $D$. If
$\delta >0$, then there exists $C=C(n,L,\delta)>0$, such that
\begin{equation}\label{eq:global}
\frac{u(x,t)}{v(x,t)}\leq C\frac{u(x_0,t_0)}{v(x_0,
  t_0)},\quad\text{for all }(x,t)\in D\cap\{t>-1+\delta ^2\}.
\end{equation}
\end{theorem}

\begin{proof}
It is an easy consequence of Theorems~\ref{thm:interior} and \ref{thm:local}.
\end{proof}

Now we show the doubling properties of the caloric measure at the
lateral boundary points by using the properties of the kernel
functions we showed in Section~\ref{sec:kernel-functions}. The idea of
the proof is similar to that of Lemma~2.2 in \cite{Wu}, but with a
more careful inspection of the different types of boundary points.  

To proceed, we will need to define the time-invariant corkscrew points
at $(y,s)$ on the lateral boundary, in addition to future and past
corkscrew points. 
Namely, for $(y,s)\in S$ we let 
\begin{alignat*}{2}
A_r(y,s)&=(y(1-r),s),&\quad&\text{if }\Psi_r(y,s)\cap E_f=\emptyset\\
A_r^\pm(y,s)&=(y'',y_{n-1}+r/2,\pm r/2,s),&&\text{if }\Psi_r(y,s)\cap
E_f\not=\emptyset.
\end{alignat*}

\begin{theorem}[Doubling at the lateral boundary points]
\label{thm:doubling}
For $0<r<1/4$ and $(y,s)\in S$ with $s\geq -1+8r^2$, there exist
$\epsilon _0=\epsilon _0(n,L)>0$ small and $C=C(n,L)>0$ such that for
any $r<\epsilon_0$ we have: 
\begin{align}
(i)\quad &\text{ If } (y,s)\in E_f \text{ and } \Psi_{2r}(y,s)\cap G_f\not=\emptyset, \text{ then }\notag\\
&C^{-1}r^nG(X,T;A_r^{\pm}(y,s))\leq \omega ^{(X,T)}(\Delta
_r(y,s))\leq Cr^n G(X,T;A_r^{\pm}(y,s))\label{eq:doubling};\\
(ii)\quad &\text{ If } (y,s)\in \cN_r(E_f)\cap \partial_p D \text{ and } \Psi_{2r}(y,s)\cap G_f=\emptyset, \text{ then }\notag\\
& C^{-1}r^nG(X,T;A_r^+(y,s))\leq \vartheta_+ ^{(X,T)}(\Delta
_r^+(y,s))\leq Cr^nG(X,T;A_r^+(y,s)),\label{eq:doubling-2}\\
& C^{-1}r^nG(X,T;A_r^-(y,s))\leq \vartheta_- ^{(X,T)}(\Delta
_r^-(y,s))\leq Cr^nG(X,T;A_r^-(y,s));\label{eq:doubling-22}\\
(iii)\quad &\text{ If } (y,s)\in \partial _pD\setminus \cN_r(E_f), \text{ then }\notag\\
& C^{-1}r^nG(X,T;A_r(y,s))\leq \omega ^{(X,T)}(\Delta
_r(y,s))\leq Cr^nG(X,T;A_r(y,s))\label{eq:doubling-3}.
\end{align}

Moreover, there is a constant $C=C(n,L)>0$, such that
\begin{align}
(i)\quad &\text{ For } (y,s)\in S\cap \{s\geq -1+8r^2\},\notag\\
&\omega ^{(X,T)}(\Delta _{2r}(y,s))\leq C\omega ^{(X,T)}(\Delta _r(y,s))u(x,t);\label{eq:doubling2}\\
(ii)\quad &\text{ For } (y,s)\in \cN_r(E_f)\cap S\cap \{s\geq -1+8r^2\},\notag\\
& \vartheta^{(X,T)}_+(\Delta_{2r}^+(y,s))\leq C \vartheta^{(X,T)}_+(\Delta_r^+(y,s)), \notag\\
& \vartheta^{(X,T)}_-(\Delta_{2r}^-(y,s))\leq C \vartheta^{(X,T)}_-(\Delta_r^-(y,s)).\label{eq:doubling2-2}
\end{align}
\end{theorem}
\begin{proof}

We start by showing the estimates from above in \eqref{eq:doubling} and
\eqref{eq:doubling-2}.

\smallskip
\case{Case 1:} $(y,s)\in E_f$ and $\Psi_{2r}(y,s)\cap G_f\neq \emptyset$. By Lemma~\ref{lem:loc-property} there is $(\tilde{y},\tilde{s})\in G_f$ such that 
$$
\Psi_{r}(y,s)\cap D\subset \Psi_{4r}(\tilde{y},\tilde{s})\cap
D\subset \Psi_{8r}(y,s)\cap D.$$ 
It is not hard to check by \eqref{eq:shift} that $F^0_r(\Delta
_{4r}(\tilde{y},\tilde{s}))\subset D$. Moreover, the parabolic
distance between $F^0_r(\Delta_{4r}(\tilde{y},\tilde{s}))$ and
$\partial _pD$, and the $t$ coordinate distance from
$F^0_r(\Delta_{4r}(\tilde{y},\tilde{s}))$ down to $A_r^\pm$ are
greater than $cr$ for some universal $c$ which only depends on $n$ and
$L$. Therefore, by the estimate of Green's function as in \cite{Wu} we
have 
\begin{equation*}
G(x,t; A_r^{\pm}(y,s))\geq C(n,L)r^{-n}, \quad (x,t)\in F^0_r(\Delta_{4r}(\tilde{y},\tilde{s}))
\end{equation*}
Applying the maximum principle to $F^0_r(D^0_r)$, we have
\begin{equation*}
G(x,t;A_r^{\pm}(y,s))\geq C(n,L)r^{-n}\omega_r^{0^{F^{0^{-1}}_r(x,t)}}(\Delta _{4r}(\tilde{y},\tilde{s})).
\end{equation*}
In particular,  
\begin{equation*}
G(X,T;A_r^{\pm}(y,s))\geq C(n,L)r^{-n}\omega_r^{0^{F^{0^{-1}}_r(X,T)}}(\Delta _{4r}(\tilde{y},\tilde{s})).
\end{equation*}
Let $(X_r,T_r):=F^{0^{-1}}_r(X,T)$ and take $(X',T')\in D$ with $T'=T-1/4$, $X'=X$ in particular $T'>1/4+T_r$. Then we obtain by the Harnack inequality that
\begin{equation}\label{eq:harnack00}
G(X,T;A_r^{\pm}(y,s))\geq C(n,L)r^{-n}\omega_r^{0^{(X',T')}}(\Delta _{4r}(\tilde{y},\tilde{s})).
\end{equation}

By Lemma~\ref{lem:kernel2}(i), for $0<r<\min\{1/4,\rho_0\}$ there exists $C=C(n,L)$ independent of $r$ such that
\begin{equation}\label{eq:harnack001}
\omega_r^{0^{(X',T')}}(\Delta_{4r}(\tilde{y},\tilde{s}))\geq C \omega ^{(X',T')}(\Delta_{4r}(\tilde{y},\tilde{s})).
\end{equation}
By Theorem~\ref{thm:kernel1} for each $(\tilde{y},\tilde{s})\in G_f$
\begin{equation*}
K_0(X',T';\tilde{y},\tilde{s})=\lim_{r\rightarrow 0} \omega ^{(X',T')}(\Delta_{4r}(\tilde{y},\tilde{s}))/\omega^{(X,T)}(\Delta_{4r}(\tilde{y},\tilde{s}))>0,
\end{equation*}
and by Corollary~\ref{cor:cont} for $(X',T')$ fixed
$K_0(X',T';\cdot,\cdot)$ is continuous on $\partial _p D$. Therefore, in the compact set $G_f$ there exists $c>0$
only depending on $n, L$ such that $K_0(X',T';\tilde{y},\tilde{s})\geq
c>0$ for any $(\tilde{y},\tilde{s})\in G_f$. Hence by the Radon-Nikodym theorem for $0<r<\min\{1/4, \rho_0\}$ we have 
\begin{equation}\label{eq:harnack01}
\omega^{(X',T')}(\Delta_{4r}(\tilde{y},\tilde{s}))\geq \frac{c}{2}\omega ^{(X,T)}(\Delta_{4r}(\tilde{y},\tilde{s}))\geq \frac{c}{2}\omega ^{(X,T)}(\Delta_r(y,s)).
\end{equation}

Combining \eqref{eq:harnack00}, \eqref{eq:harnack001} and \eqref{eq:harnack01} we obtain the
estimate from above in \eqref{eq:doubling} for Case~1.

\smallskip
\case{Case 2:} $(y,s)\in \cN_r(E_f)\cap \partial_pD$ and $\Psi_{2r}(y,s)\cap
G_f=\emptyset$. 

In this case $\Psi_{2r}(y,s)\cap D$ splits into a disjoint union of $\Psi_{2r}(y,s)\cap D_\pm$. 
We use $F^+_r$ and $F^-_r$ defined in \eqref{eq:shift11} and \eqref{eq:shift22}, and apply the same arguments as in Case 1 in $D_r^+$ and $D_r^-$. Then 
\begin{equation*}
\omega ^{\pm^{(X,T)}}_{r}(\Delta _{r}^\pm(y,s))\leq C r^nG(X,T; A_r^\pm(y,s)).
\end{equation*}
Taking $0<r<\delta _0$, where $\delta_0=\delta_0(n,L)$ is the constant
in Lemma~\ref{lem:kernel2}(ii), we have
\begin{equation*}
\vartheta^{(X,T)}_\pm(\Delta _r (y,s))\leq 2\omega^{\pm{(X,T)}}_r(\Delta_r(y,s))\leq C r^nG(X,T; A_r^\pm(y,s)).
\end{equation*}

\smallskip
\case{Case 3:} $(y,s)\in \partial_pD\setminus \cN_r(E_f)$. We argue similarly to Case 1 and 2.

\smallskip
Taking $\epsilon_0=\min\{\rho_0,\delta_0,1/4\}$, we complete the proof
of the estimates from above in \eqref{eq:doubling}--\eqref{eq:doubling-3}.

The proof of the estimate from below in \eqref{eq:doubling}--\eqref{eq:doubling-3} is the same as in \cite{Wu}. For \eqref{eq:doubling} it is a consequence
of Lemma~\ref{lem:caloric1} and the maximum principle. \eqref{eq:doubling-2} and \eqref{eq:doubling-22} follow from \eqref{eq:theta-lemma2} and the maximum principle. The doubling properties of caloric measure $\omega^{(x,t)}$ and $\theta_\pm^{(x,t)}$ are easy
consequences of \eqref{eq:doubling}--\eqref{eq:doubling-3} and Proposition~\ref{prop:limitfunc}(ii) for $0<r<\epsilon_0/2$. For
$r>\epsilon_0/2$ we use Lemma~\ref{lem:caloric1} and \eqref{eq:theta-lemma2}. 

\end{proof}

Theorem~\ref{thm:doubling} implies the following
backward Harnack principle.

\begin{theorem}[Backward boundary Harnack principle]
\label{thm:backward}
Let $u$ be a positive caloric function in $D$ vanishing continuously
on $S$ and let $\delta >0$. Then there exists a positive constant
$C=C(n,L, \delta)$ such that for $(y,s)\in \partial_pD\cap
\{s>-1+\delta^2\}$ and for $0<r<r(n,L,\delta)$ sufficiently small we have
\begin{align*}
C^{-1}u(\underline{A}^+_r(y,s))&\leq u(\overline{A}^+_r(y,s))\leq Cu(\underline{A}^+_r(y,s)),\\
C^{-1}u(\underline{A}^-_r(y,s))&\leq u(\overline{A}^-_r(y,s))\leq Cu(\underline{A}^-_r(y,s)),\quad \text{if }(y,s)\in \cN_r(E_f);
\end{align*}
and
\begin{equation}
C^{-1}u(\underline{A}_r(y,s))\leq u(\overline{A}_r(y,s))\leq Cu(\underline{A}_r(y,s)),\quad \text{if }(y,s)\not\in \cN_r(E_f).\label{eq:backward33}
\end{equation}

\end{theorem}

\begin{proof}
Once we have Theorem~\ref{thm:doubling}, which is an analogue of Lemma~2.2 in \cite{Wu}, we can proceed as Theorem 4 in \cite{Garofalo3} to show the above backward Harnack principle.

\end{proof}

\begin{remark}\label{rem:equivalent}
From \eqref{eq:doubling} and using the same proof as in Theorem~\ref{thm:backward} we can conclude that for any positive caloric function $u$ vanishing continuously on $S$ and $(y,s)\in G_f$ there exists 
$C=C(n,L,\delta)>0$ such that
\begin{align*}
& C^{-1} u(\overline{A}_r^-(y,s))\leq u(\overline{A}_r^+(y,s))\leq C u(\overline{A}_r^-(y,s)),\\
& C^{-1} u(\underline{A}_r^-(y,s))\leq u(\underline{A}_r^+(y,s))\leq C u(\underline{A}_r^-(y,s)).
\end{align*}
\end{remark}

\section{Various versions of boundary Harnack}
\label{sec:applications}
In the applications,
 it is very useful to have a local version of the backward
Harnack for solutions vanishing only on a portion of the lateral
boundary $S$. For the parabolically Lipschitz domains this was proved
in \cite{ACS} as a consequence of the (global) backward Harnack principle.

To state the results, we will use the following corkscrew point associated with
$(y,s)\in G_f$: for $0<r<1/4$, let 
\begin{align*}
& \overline{A}_r(y,s)=(y'',y_{n-1}+4nLr,0,s+2r^2),\\
& \underline{A}_r(y,s)=(y'',y_{n-1}+4nLr,0,s-2r^2),\\
& A_r(y,s)=(y'',y_{n-1}+4nLr,0,s).
\end{align*}
When $(y,s)=(0,0)$ we simply write $\overline{A}_r$, $\underline{A}_r$
and $A_r$, in addition to $\Psi_r$, $\Delta_r$, $\overline{A}^\pm_r$, $\underline{A}^\pm_r$.

\begin{theorem}\label{thm:backward2}
Let $u$ be nonnegative caloric in $D$, continuously vanishing
continuously on $E_f$. Let $m=u(\underline{A}_{3/4})$, $M=\sup _D u$, then there exists a constant $C=C(n, L, M/m)$, such that for any $0<r<1/4$ we have \begin{equation}\label{eq:vari1}
u(\overline{A}_r)\leq Cu(\underline{A}_r).
\end{equation}
\end{theorem}
\begin{proof}
Using Theorems~\ref{thm:backward} and \ref{thm:doubling} and following the lines of Theorem~13.7 in \cite{CS} we have
\begin{equation*}
u(\overline{A}^\pm_{2r})\leq C u(\underline{A}^\pm_{2r}), \quad
0<r<1/4,
\end{equation*}
for $C=C(n,L,M/m)$. Then \eqref{eq:vari1} follows from Theorem~\ref{thm:backward} and an
observation that there is a Harnack chain with a constant
$\mu=\mu(n,L)$ and length $N=N(n,L)$ joining $\overline{A}_r$ to $\overline{A}_{2r}^\pm$ and $\underline{A}_{2r}^\pm$ to $\underline{A}_r$.
\end{proof}

Theorem~\ref{thm:backward2} implies the boundary H\"older regularity
of the quotient of two negative caloric functions vanishing on
$E_f$. The proof of the following corollary is the same as for Corollary~13.8 in \cite{CS} and is therefore omitted.

\begin{theorem}\label{thm:backward3}
Let $u_1$, $u_2$ be nonnegative caloric functions in $D$ continuously vanishing
 on $E_f$.  Let $M_i=\sup _D u_i$ and
$m_i=u_i(\underline{A}_{3/4})$
with $i=1,2$. Then we have 
\begin{equation}\label{eq:backward3}
C^{-1}\frac{u_1(A_{1/4})}{u_2(A_{1/4})}\leq
\frac{u_1(x,t)}{u_2(x,t)}\leq
C\frac{u_1(A_{1/4})}{u_2(A_{1/4})},\quad\text{for
}(x,t)\cap\Psi_{1/8}\cap D,
\end{equation}
where $C=C(n,L, M_1/m_1, M_2/m_2)$. Moreover, if $u_1$ and $u_2$ are symmetric
in $x_n$, then $u_1/u_2$ extends to a
function in $C^{\alpha}(\Psi_{1/8})$ for some $0<\alpha <1$, where the
exponent $\alpha$ and the $C^{\alpha}$ norm depend only on $n, L, M_1/m_1,
M_2/m_2$.\qed 
\end{theorem}
\begin{remark} The symmetry condition in the latter part of the
  theorem is important to guarantee the continuous extension of
  $u_1/u_2$ to the Euclidean closure $\overline{\Psi_{1/8}\setminus
    E_f}=\overline{\Psi_{1/8}}$, 
  since the limits at $E_f\setminus G_f$, as we approach from different
  sides, may be different. Without the symmetry condition, one may
  still prove that $u_1/u_2$ extends to a $C^\alpha$ function on the
  completion $(\Psi_{1/8}\setminus E_f)^*$ with respect to the inner
  metric.
  \end{remark}

For a more general application, we need to have a boundary Harnack
inequality for $u$ satisfying a nonhomogeneous equation with bounded
right hand side but additionally with a nondegeneracy condition. The
method we use here is similar as the one used in the elliptic case
(\cite{CSS}). 

\begin{theorem}\label{thm:nondegeneracy}
Let $u$ be a nonnegative function in $D$, continuously
vanishing on $E_f$, and satisfying  
\begin{align}
|\Delta u-\partial _tu|\leq C_0&\quad\text{in }D\\
\label{eq:nondeg} 
u(x,t)\geq c_0d(x,t)^\gamma&\quad\text{in }D,
\end{align}
where $d(x,t)=\dist_p((x,t);E_f)$, $0<\gamma<2$, $c_0>0$,
$C_0\geq0$. Then there exists $C=C(n,L,\gamma,C_0,c_0)>0$ such that for
$0<r<1/4$ we have
\begin{equation}\label{eq:nondegeneracy}
u(x,t)\leq Cu(\overline{A}_r),\quad (x,t)\in \Psi _{r}.
\end{equation}
Moreover, if $M=\sup _D u$, then there exists a constant $C=C(n, L,
\gamma, C_0, c_0, M)$, such that for any $0<r<1/4$ we have 
\begin{equation}\label{eq:nondeg-vari1}
u(\overline{A}_r)\leq Cu(\underline{A}_r).
\end{equation}

\end{theorem}

\begin{proof}
Let $u^{\ast}$ solve the heat equation in $\Psi _{2r}\cap D$ and equal
to $u$ on $\partial _p(\Psi_{2r}\cap D)$. Then by the Carleson
estimate we have $u^{\ast}(x,t)\leq C(n,L)u^{\ast}(\overline{A}_r)$
for $(x,t)\in \Psi_r$. 

On the other hand, we have
\begin{alignat*}{2}
u^{\ast}(x,t)+C(|x|^2-t-8r^2)&\leq u(x,t)&&\quad\text{on }\partial
_p(\Psi_{2r}\cap D)\\
(\Delta -\partial _t)(u^{\ast}(x,t)+C(|x|^2-t-8r^2))&\geq C(2n-1)\\
&\geq (\Delta -\partial _t)u(x,t)&&\quad\text{in }\Psi _{2r}\cap D
\end{alignat*}
for $C\geq C_0/(2n-1)$. Hence, by the comparison principle we have $u^{\ast}-u\leq Cr^2$ in $\Psi_{2r}\cap D$ for $C=C(C_0,n)$.
Similarly, $u-u^{\ast}\leq Cr^2$ and hence 
$|u-u^{\ast}|\leq Cr^2$ in $\Psi_{2r}\cap D$. Consequently,
\begin{equation}\label{eq:nondegg}
u(x,t)\leq C(n,L)(u(\overline{A}_r)+C(C_0,n)r^2),\quad (x,t)\in\Psi_r.
\end{equation}
Next note that by the nondegeneracy condition \eqref{eq:nondeg}
\begin{equation}\label{eq:nondeggg}
u(\overline{A}_r)\geq c_0r^\gamma \geq c_0r^2, \quad r\in (0,1).
\end{equation}
Thus, combining \eqref{eq:nondegg} and
\eqref{eq:nondeggg}, we obtain  \eqref{eq:nondegeneracy}. 

The proof of \eqref{eq:nondeg-vari1} follows in a similar manner from
Theorem~\ref{thm:backward2} for $u_*$.
\end{proof}

\begin{remark} In fact, the nondegeneracy condition \eqref{eq:nondeg}
  is necessary. An easy counterexample is $u(x,t)=x_{n-1}^2x_n^2$ in
  $\Psi_1$ and $E_f=\{(x,t):x_{n-1}\leq 0,x_{n}=0 \}\cap \Psi_1$. Then
  $u(\overline{A}_r)=0$ for $r\in (0,1)$ but obviously $u$ does not
  vanish in $\Psi_r\cap D$. 
\end{remark}

We next state a generalization of the local comparison theorem.

\begin{theorem}\label{thm:nondegeneracy2}
Let $u_i$, $i=1,2$, be nonnegative functions in $D$, continuously
vanishing on $E_f$, and satisfying  
\begin{align*}
|\Delta u_i-\partial _tu_i|\leq C_0&\quad\text{in }D\\
u_i(x,t)\geq c_0d(x,t)^\gamma&\quad\text{in }D,
\end{align*}
where $d(x,t)=\dist_p((x,t);E_f)$, $0<\gamma<2$, $c_0>0$,
$C_0\geq0$.
Let also $M=\max\{\sup_D u_1,\sup_D
u_2\}$. Then there exists a 
constant $C=C(n, L,\gamma, C_0,c_0, M)>0$ such that 
\begin{equation}\label{eq:nondegeneracy2}
C^{-1}\frac{u_1(A_{1/4})}{u_2(A_{1/4})}\leq \frac{u_1(x,t)}{u_2(x,t)}\leq C\frac{u_1(A_{1/4})}{u_2(A_{1/4})},\quad (x,t)\in \Psi _{1/8}\cap D.
\end{equation}
Moreover, if $u_1$ and $u_2$ are symmetric
in $x_n$, then $u_1/u_2$ extends to a function in
$C^{\alpha}(\overline{\Psi_{1/8}})$ for some $0<\alpha <1$, with
$\alpha$ and  $C^{\alpha}$ norm  depending 
only on $n, L, \gamma, C_0,c_0, M$. 
\end{theorem}

To prove this theorem, we will also need the following two lemmas, which
are essentially Lemmas~11.5 and 11.8 in \cite{DGPT}. The proofs are therefore omitted.

\begin{lemma}\label{lem:signor-glob-local} Let $\Lambda$ be a subset of
  $\R^{n-1}\times(-\infty,0]$, and $h(x,t)$ a continuous function
  in $\Psi_{1}$. Then for any $\delta_0>0$
  there exists $\epsilon_0>0$ depending only on $\delta_0$ and $n$
  such that if
\begin{enumerate}
\item[i)] $h\geq 0$ on $\Psi_1\cap\Lambda$,
\item[ii)] $(\Delta-\partial_t)h\leq \epsilon_0$ in $\Psi_1\setminus\Lambda$,
\item[iii)] $h\geq-\epsilon_0$ in $\Psi_1$,
\item[iv)] $h\geq\delta_0$ in $\Psi_1\cap\{|x_{n}|\geq \beta_n\}$, $\beta_n=1/(32\sqrt{n-1})$
\end{enumerate}
then $h\geq 0$ in $\Psi_{1/2}$.\qed
\end{lemma}
\begin{lemma}\label{lem:h-nondeg}
\pushQED{\qed}
For any $\delta_0>0$ there exists $\epsilon_0>0$ and
  $c_0>0$ depending only on $\delta_0$ and $n$ such that if $h$ is a
  continuous function on $\Psi_1\cap \{0\leq x_{n}\leq\beta_n\}$, $\beta_n=1/(32\sqrt{n-1})$, satisfying
\begin{enumerate}
\item[i)] $(\Delta -\partial_t)h\leq \epsilon_0$ in $\Psi_1\cap\{0<x_{n}<\beta_n\}$
\item[ii)] $h\geq 0$ in  $\Psi_1\cap\{0<x_{n}<\beta_n\}$,
\item[iii)] $h\geq \delta_0$ on $\Psi_1\cap\{x_{n}=\beta_n\}$,
\end{enumerate}
then
\[
h(x,t)\geq c_0 x_{n}\quad\text{in }\Psi_{1/2}\cap\{0<x_{n}<\beta_n\}.\qedhere
\]
\popQED
\end{lemma}

\begin{proof}[Proof of Theorem~\ref{thm:nondegeneracy2}]
We first note that arguing as in the proof of
Theorem~\ref{thm:nondegeneracy} and using Theorem~\ref{thm:backward2}, we
will have that 
\begin{equation}\label{eq:rho0}
u_i(x,t)\leq C u_i(A_{1/4}), \quad (x,t)\in \Psi_{1/8},
\end{equation}
for $C=C(n, L,\gamma, C_0,c_0, M)$. 
Next, dividing $u_i$ by $u_i(A_{1/4})$, we can assume
$u_i(A_{1/4})=1$. Then, consider the rescalings  
$$
u_{i\rho}(x,t)=\frac{u_i(\rho x, \rho^2 t)}{\rho^\gamma}, \quad \rho\in (0,1),\quad i=1,2.$$
It is immediate to verify that $u_{i\rho}$ satisfy for $(x,t)\in
\Psi_{1/(8\rho)}\cap D$, 
\begin{align}
& |(\Delta -\partial _t) u_{i\rho}(x,t)|\leq C_0 \rho^{2-\gamma}, \label{eq:rho1}\\
& u_{i\rho}(x,t)\geq c_0  \dist((x,t), E_{f_\rho})^\gamma,\label{eq:rho2}\\
& u_{i\rho}(x,t)\leq \frac{C}{\rho^\gamma}, \quad C \text{ is the constant in \eqref{eq:rho0}}\label{eq:rho3},
\end{align}
where $f_\rho(x'',t)=(1/\rho)f(\rho x'',\rho^2 t)$ is the scaling of $f$.
By \eqref{eq:rho2} there exists $c_n>0$ such that 
\begin{equation}\label{eq:rho4}
u_{i\rho}(x,t)\geq c_0c_n,\quad (x,t)\in \Psi_{1/(8\rho)}\cap \{|x_{n}|\geq \beta_n\}. 
\end{equation}
Consider now the difference 
$$
h=u_{2\rho}-su_{1\rho},
$$ 
for a small positive $s$, specified below. By \eqref{eq:rho1},
\eqref{eq:rho4}, \eqref{eq:rho3} one can choose a positive
$\rho=\rho(n, L,\gamma, C_0,c_0, M)<1/16$ and $s=s(\rho,n,c_0,C)>0$ such
that   
\begin{alignat*}{2}
& h(x,t)\geq c_0c_n-s \cdot \frac{C}{\rho^\gamma}\geq \frac{c_0c_n}{2},&\quad& (x,t)\in \Psi_{1/(8\rho)}\cap \{|x_{n}|\geq \beta_n\},\\
& h(x,t) \geq -s \cdot \frac{C}{\rho^\gamma}\geq -\epsilon_0,&\quad& (x,t)\in \Psi_{1/(8\rho)},\\
& |(\Delta-\partial_t)h(x,t)|\leq  C_0 \rho^{2-\gamma}\leq \epsilon_0,&\quad& (x,t)\in \Psi_{1/(8\rho)}\cap D,
\end{alignat*}
where $\epsilon_0=\epsilon_0(c_0,c_n, n)$ is the constant in Lemma~11.5. Thus by Lemma~11.5, $h>0$ in $\Psi_{1/2}\cap D$, which implies
\begin{equation}\label{eq:rho5}
\frac{u_1(x,t)}{u_2(x,t)}\leq \frac{1}{s},\quad (x,t)\in \Psi_{\rho/2}\cap D.
\end{equation}
By moving the origin to any $(z,h)\in \Psi_{1/8}\cap E_f$ we will
therefore obtain the bound
\begin{equation}\label{eq:rho5-1}
\frac{u_1(x,t)}{u_2(x,t)}\leq C(n, L,\gamma, C_0,c_0, M)
\end{equation}
for any $(x,t)\in \Psi_{1/8}\cap \cN_{\rho/2}(E_f)\cap D$. On the other
hand, for $(x,t)\in \Psi_{1/8}\setminus \cN_{\rho/2}(E_f)$ the
estimate \eqref{eq:rho5-1} will follow from \eqref{eq:nondeg} and
\eqref{eq:rho0}.
Hence \eqref{eq:rho5-1} holds for any $(x,t)\in \Psi_{1/8}\cap D$, which gives the bound from above in
\eqref{eq:nondegeneracy2}. Changing the roles of $u_1$ and $u_2$ we
get the bound from below.

The proof of $C^\alpha$ regularity follows by iteration from
\eqref{eq:nondegeneracy2} similarly to the proof of Corollary~13.8 in
\cite{CS}; however, we need to make sure that at every
step the nondegeneracy condition is satisfied. We will only verify the H\"older continuity of $u_1/u_2$ at the origin, the
rest being standard.

For $k\in\N$ and $\lambda>0$ to be specified below let
$$
l_k=\inf_{\Psi_{\lambda^k}\cap D}\frac{u_1}{u_2},\quad L_k=\sup_{\Psi_{\lambda^k}\cap D}\frac{u_1}{u_2}.
$$
We then know that $1/C\leq l_k\leq L_k\leq C$ for $\lambda\leq
1/8$. Let also 
$$
\mu_k=\frac{u_1(\underline A_{\lambda^k/4})}{u_2(\underline A_{\lambda^k/4})}\in [l_k,L_k].
$$
Then there are two possibilities: 
$$
\text{either}\quad L_k-\mu_k\geq \frac12
(L_k-l_k)\quad\text{or}\quad \mu_k-l_k\geq \frac12
(L_k-l_k).
$$
For definiteness, assume that we are in the latter case, the former
cases being treated similarly.
Then consider two functions
$$
v_1(x,t)=\frac{u_1(\lambda^k x,\lambda^{2k} t)-l_k u_2(\lambda^k
  x,\lambda^{2k}t)}{u_1(\underline A_{\lambda^k/4})-l_k u_2(\underline
  A_{\lambda^k/4})},\quad
v_2(x,t)=\frac{u_2(\lambda^kx,\lambda^{2k}t)}{u_2(\underline A_{\lambda^k/4})}.
$$
In $\Psi_1\setminus E_{f_{\lambda^k}}$, we will have
\begin{align*}
|(\Delta-\partial_t) v_1(x,t)|&\leq
\frac{\lambda^{2k}(1+l_k)C_0}{u_1(\underline A_{\lambda^k/4})-l_k
  u_2(\underline A_{\lambda^k/4})},\\
|(\Delta-\partial_t) v_2(x,t)|&\leq
\frac{\lambda^{2k}C_0}{u_2(\underline A_{\lambda^k/4})}.
\end{align*}
To proceed, fix a small $\eta_0>0$, to be specified below. Then from the nondegeneracy of $u_2$, we immediately have
$$
|(\Delta-\partial_t) v_2(x,t)|\leq C\lambda^{(2-\gamma)k}<\eta_0,
$$
if we take $\lambda$ small enough. For $v_1$, we have a dichotomy:
$$
\text{either}\quad |(\Delta-\partial_t) v_1(x,t)|\leq
\eta_0\quad\text{or}\quad \mu_k-l_k\leq C\lambda^{(2-\gamma)k}.
$$
In the latter case, we obtain
\begin{equation}
\label{eq:nondeg-iter-1}
L_k-l_k\leq 2(\mu_k-l_k)\leq C\lambda^{(2-\gamma)k}.
\end{equation}
In the former case we notice that both functions $v=v_1$, $v_2$
satisfy 
$$
v\geq 0,\quad v(\underline A_{1/4})=1\quad\text{and}\quad |(\Delta-\partial_t) v(x,t)|\leq
\eta_0\quad\text{in }\Psi_1\setminus E_{f_{\lambda^k}} 
$$
and that $v$ vanishes continuously on $\Psi_1\cap E_{f_{\lambda^k}}$.
We next establish a nondegeneracy property for such $v$.
Indeed, first note that by the parabolic Harnack
inequality, see Theorems 6.17 and
6.18 in \cite{Lieberman}, for small enough $\eta_0$, we will have that
$$
v\geq c_n\quad\text{on }\Psi_{1/8}\cap \{|x_n|\geq \beta_n/8\}.
$$
Then, by invoking Lemma~\ref{lem:h-nondeg}, we will obtain that
\begin{equation}
\label{eq:nondeg-iter-2}
v(x,t)\geq c_n|x_n|\quad\text{in }\Psi_{1/16}\setminus E_{f_{\lambda^k}}.
\end{equation}
We further claim that
\begin{equation}
\label{eq:nondeg-iter-3}
v(x,t)\geq c\dist_p((x,t), E_{f_{\lambda^k}})\quad\text{in }\Psi_{1/32}\setminus E_{f_{\lambda^k}}.
\end{equation}
To this end, for $(x,t)\in \Psi_{1/32}\setminus E_{f_{\lambda^k}}$ let $d=\sup\{r:\Psi_r(x,t)\cap
E_{f_{\lambda^k}}=\emptyset\}$ and consider the box
$\Psi_d(x,t)$. Without loss of generality assume $x_n\geq 0$. Then let
$(x_*,t_*)=(x',x_n+d, t-d^2)\in\partial_p\Psi_d(x,t)$. From
\eqref{eq:nondeg-iter-2} we have that
$$
v(x_*,t_*)\geq c_n (x_n+d)\geq c_nd
$$
and applying the parabolic Harnack inequality, we obtain
$$
v(x,t)\geq c_n v(x_*,t_*)-C_n \eta_0 d^2\geq c_n d,
$$
provided $\eta_0$ is sufficiently small. Hence,
\eqref{eq:nondeg-iter-3} follows.

Having the nondegeneracy, we also have the bound from above for
functions $v_1$ and $v_2$. Indeed, by Theorem~\ref{thm:nondegeneracy}
for $v_1$ and $v_2$ we have
\begin{multline}
\label{eq:nondeg-iter-4}
\sup_{\Psi_1} v_1\leq C v_1(\overline A_{1/4})=
C\frac{u_1(\overline A_{\lambda^k/4})-l_k u_2(\overline
  A_{\lambda^k/4})}{u_1(\underline A_{\lambda^k/4})-l_k u_2(\underline
  A_{\lambda^k/4})}\\
\leq C \frac{u_2(\overline A_{\lambda^k/4})}{u_2(\underline
  A_{\lambda^k/4})}\frac{L_k-l_k}{\mu_k-l_k}\leq C
\end{multline}
and
\begin{equation}\label{eq:nondeg-iter-5}
\sup_{\Psi_1} v_2  \leq C v_2(\overline A_{1/4})=C \frac{u_2(\overline A_{\lambda^k/4})}{u_2(\underline
  A_{\lambda^k/4})}\leq C,
\end{equation}
where we have also invoked the second part of
Theorem~\ref{thm:nondegeneracy} for $u_2$.

We thus verified all conditions necessary for applying the estimate \eqref{eq:nondegeneracy2} to functions
$v_1$ and $v_2$. Particularly, the inequality from below, applied in $\Psi_{8\lambda}\setminus
  E_{f_{\lambda^k}}$, will give
$$
\inf_{\Psi_{\lambda}\setminus
  E_{f_{\lambda^k}}}\frac{v_1}{v_2}\geq c\frac{v_1(A_{2\lambda})}{v_2(A_{2\lambda})}\geq c\lambda
$$
for a small $c>0$, or equivalently 
$$
l_{k+1}-l_k\geq c\lambda (\mu_k-\l_k)\geq \frac{c\lambda}2 (L_k-l_k).
$$
Hence, we will have
\begin{equation}\label{eq:nondeg-iter-6}
L_{k+1}-l_{k+1}\leq L_k-\l_{k}-(\l_{k+1}-l_k)\leq \left(1-\frac{c\lambda}2\right)(L_k-l_k).
\end{equation}
Summarizing,  \eqref{eq:nondeg-iter-1} and \eqref{eq:nondeg-iter-6}
give a dichotomy: for any $k\in\N$,
$$
\text{either}\quad L_k-l_k\leq C
\lambda^{(2-\gamma)k}\quad\text{or}\quad L_{k+1}-l_{k+1}\leq (1-c\lambda/2)(L_k-l_k).
$$
This clearly implies that
$$
L_k-l_k\leq C\beta^k\quad\text{for some }\beta\in (0,1),
$$
for any $k\in\N$, which is nothing but the H\"older continuity of
$u_1/u_2$ at the origin. 
\end{proof}

We next want to prove a variant of Theorem~\ref{thm:nondegeneracy2} but
with $\Psi_r$ replaced with their lower halves
$$
\Theta_r=\Psi_r\cap\{t\leq 0\}.
$$
\begin{theorem}\label{thm:nondegeneracy2-lower}
Let $u_i$, $i=1,2$, be nonnegative functions in $\Theta_1\setminus E_f$,
continuously vanishing on $\Theta_1\cap E_f$, and satisfying  
\begin{align*}
|\Delta u_i-\partial _tu_i|\leq C_0&\quad\text{in }\Theta_1\setminus E_f\\
u_i(x,t)\geq c_0\dist((x,t),E_f)&\quad\text{in }\Theta_1\setminus E_f,
\end{align*}
for some $c_0>0$, $C_0\geq0$.
Let also $M=\max\{\sup_D u_1,\sup_D
u_2\}$. 
Moreover, if $u_1$ and $u_2$ are symmetric
in $x_n$, then $u_1/u_2$ extends to a function in
$C^{\alpha}(\overline\Theta_{1/8})$ for some $0<\alpha <1$, with
$\alpha$ and  $C^{\alpha}$ norm  depending only on $n, L, \gamma, C_0,c_0, M$.
\end{theorem}

The idea is that the functions $u_i$ can be extended to
$\Psi_\delta$, for some $\delta>0$, while still keeping the same
inequalities, including the nondegeneracy condition.

\begin{lemma}\label{lem:nondeg-extension} Let $u$ be a nonnegative continuous function on
  $\Theta_1$ such that
\begin{align*} 
u=0&\quad\text{in }\Theta_1\cap E_f\\
|(\Delta-\partial_t) u|\leq C_0&\quad\text{in }\Theta_1\setminus E_f\\
u(x,t)\geq c_0\dist_p(x,t;E_f)&\quad\text{in }\Theta_1\setminus E_f. 
\end{align*}
for some $C_0\geq 0$, $c_0>0$.
Then, there exists positive $\delta$ and $\tilde c_0$ depending only
on $n$, $L$, $c_0$ and $C_0$, and a nonnegative extension $\tilde u$ of $u$ to
$\Psi_\delta$ such that
\begin{align*} 
\tilde u=0&\quad\text{in }\Psi_\delta\cap E_f\\
|(\Delta-\partial_t) \tilde u|\leq C_0&\quad\text{in }\Psi_\delta\setminus E_f\\
\tilde u(x,t)\geq \tilde c_0\dist_p((x,t),E_f)&\quad\text{in }\Psi_\delta\setminus E_f. 
\end{align*}
Moreover, we will also have that $\sup_{\Psi_\delta}\tilde u\leq \sup_{\Theta_1}u$.
\end{lemma}
\begin{proof} We first continuously extend the function $u$ from the
  parabolic boundary $\partial_p\Theta_{1/2}$ to
  $\partial_p\Psi_{1/2}$ by also keeping it nonnegative and bounded
  above by the same constant. Further, put $u=0$
  on $E_f\cap(\Psi_{1/2}\setminus \Theta_{1/2})$. Then extend $u$ to $\Psi_{1/2}$ by solving the Dirichlet problem for the heat
  equation in $(\Psi_{1/2}\setminus \Theta_{1/2})\setminus E_f$, with
  already defined boundary values. We still denote it the extended
  function by $u$. 

Then it is easy to see that $u$ is nonnegative in $\Psi_{1/2}$,
$\sup_{\Psi_{1/2}}u\leq\sup_{\Theta_1} u$, $u$
vanishes on $\Psi_{1/2}\cap E_f$ and $|(\Delta-\partial_t) u|\leq
C_0$ in $\Psi_{1/2}\setminus E_f$.  Note that we still have the
nondegeneracy property  $u(x,t)\geq c_0 \dist_p((x,t),E_f)$ for in
$\Theta_{1/2}\setminus E_f$, so it remains to prove the nondegeneracy
for $t\geq 0$. We will be able to do it in a small box $\Psi_\delta$,
as a consequence of Lemma~\ref{lem:h-nondeg}.

For $0<\delta<1/2$ consider the rescalings
$$
u_\delta(x,t)=\frac{u(\delta x,\delta^2 t)}{\delta},\quad (x,t)\in \Psi_{1/(2\delta)}.
$$
Then we have
\begin{align*}
|(\Delta -\partial_t)u_\delta|\leq C_0\delta,&\quad\text{in
}\Psi_1\setminus E_{f_\delta}\\
u_\delta(x,t)\geq c_0|x_{n}|&\quad\text{in }\Theta_1,
\end{align*}
where $f_\delta(x'',t)=(1/\delta)f(\delta x'',\delta^2 t)$ is the
rescaling of $f$.
Then by using the parabolic Harnack inequality (see Theorems 6.17 and
6.18 in \cite{Lieberman}) in $\Theta_1^\pm$, we
obtain that 
$$
u_\delta(x,t)\geq c_n c_0-C_nC_0\delta> c_1 \quad\text{on
}\{|x_{n}|=\beta_n/2\}\cap \Psi_{1/2}.
$$
Further, choosing $\delta$ small and applying Lemma~\ref{lem:h-nondeg},
we deduce that
$$
u_\delta(x,t)\geq c_2 |x_{n}|\quad\text{in } \Psi_{1/4}.
$$
Then, repeating the arguments based on the parabolic Harnack
inequality, as for the inequality \eqref{eq:nondeg-iter-3}, we obtain
\[
u(x,t)\geq C \dist_p((x,t),E_{f_\delta}),\quad \text{in }\Psi_{1/8}.
\]
Scaling back, this gives
\[
u(x,t)\geq C \dist_p((x,t),E_{f}),\quad \text{in }\Psi_{\delta/8}.\qedhere
\]
\end{proof}
\begin{proof}[Proof of Theorem~\ref{thm:nondegeneracy2-lower}]
Extend functions $u_i$ as is Lemma~\ref{lem:nondeg-extension} and
apply Theorem~\ref{thm:nondegeneracy2}. If we repeat this at every
$(y,s)\in \Theta_{1/8}\cap G_f$, we will obtain the H\"older
regularity of $u_1/u_2$ in $\cN_{\delta/8}(\Theta_{1/8}\cap
G_f)\cap\{t\leq 0\}$. For
the remaining part of $\Theta_{1/8}$, we argue as in the proof of
localization property Lemma~\ref{lem:loc-property} cases 1) ans 2),
and use the corresponding results for parabolically Lipschitz domains.
\end{proof}

\subsection{Parabolic Signorini problem} 
In this subsection we discuss an application of the boundary Harnack
principle in the parabolic Signorini problem. The idea of such
applications goes back to the paper Athanasopoulos and Caffarelli
\cite{AC}. The particular result that we will discuss here, can be found
also in \cite{DGPT}, with the same proof based on our
Theorem~\ref{thm:nondegeneracy2-lower}.

In what follows, we will use $H^{\ell,\ell/2}$, $\ell>0$, to denote the parabolic
H\"older classes, as defined for instance in \cite{LSU}.

For a given function $\phi\in H^{\ell,\ell/2}(Q_1')$, $\ell\geq 2$,
known as the \emph{thin obstacle}, we say that a function $v$ solves the
\emph{parabolic Signorini problem} if $v\in W^{2,1}_2(Q_1^+)\cap
H^{1+\alpha,(1+\alpha)/2}(\overline{Q_1^+})$ , $\alpha>0$, and
\begin{align}
\label{eq:par-sig-1}(\Delta -\partial _t)v=0 &\quad \text{in } Q_1^+,\\
\label{eq:par-sig-2}v\geq \phi,\quad -\partial _{x_{n}}v\geq 0,\quad (v-\phi)\partial_{x_{n}}v=0&\quad \text{on } Q'_1.
\end{align}
This kind of problems appears in many applications, such as thermics
(boundary heat control), 
biochemistry (semipermeable membranes and osmosis), and elastostatics
(the original Signorini problem). We refer to the book \cite{DL} for
the derivation of such models as well as for some basic existence and
uniqueness results. 

The regularity that we impose on the solutions
\eqref{eq:par-sig-1}--\eqref{eq:par-sig-2} is also well known in the
literature, see e.g.\ \cites{Ath1,Ur,AU}. It was proved recently in \cite{DGPT} that one can actually take
$\alpha=1/2$ in the regularity assumptions on $v$, which is the
optimal regularity as can be seen from the explicit example
$$
v(x,t)=\Re(x_{n-1}+i x_n)^{3/2},
$$
which solves the Signorini problem with $\phi=0$. One of the main
objects of study in the Signorini problem is the \emph{free boundary}
$$
G(v)=\partial_{Q_1'}(\{v>\phi\}\cap Q_1'),
$$
where $\partial_{Q_1'}$ is the boundary in the relative topology of
$Q_1'$. 

As the initial step in the study, we make the following reduction. We
observe that the difference  
$$
u(x,t)=v(x,t)-\phi(x',t)
$$
will satisfy
\begin{align}
\label{eq:par-sig-1-nonhom}(\Delta -\partial _t)u=g &\quad \text{in } Q_1^+,\\
\label{eq:par-sig-2-nonhom}u\geq 0,\quad -\partial _{x_{n}}u\geq 0,\quad u\partial_{x_{n}}u=0&\quad \text{on } Q'_1,
\end{align}
where $g=-(\Delta_{x'}-\partial_t)\phi\in H^{\ell-2,(\ell-2)/2}$.  That
is, one can make the thin obstacle equal to $0$ at the expense
of getting a nonzero right-hand side in the equation for $u$.
For our purposes, this simple reduction will be sufficient, however, to
take the full advantage of the regularity of $\phi$, when $\ell>2$, one may need to
subtract an additional polynomial from $u$ to guarantee the decay
rate
$$
|g(x,t)|\leq M(|x|^2+|t|)^{(\ell-2)/2}
$$
near the origin, see Proposition~4.4 in \cite{DGPT}.
With the reduction above, the free
boundary $G(v)$ becomes
$$
G(u)=\partial_{Q_1'}(\{u>0\}\cap Q_1').
$$
Further, it will be convenient to consider the even
extension of $u$ in $x_{n-1}$ variable to the entire $Q_1$, i.e., by
putting $u(x',x_n,t)=u(x',-x_{n},t)$. Then such an
extended function will satisfy
$$
(\Delta-\partial_t)u=g\quad\text{in }Q_1\setminus\Lambda(u), 
$$
where $g$ has also been extended by even symmetry in $x_n$, and where
$$
\Lambda(u)=\{u=0\}\cap Q_1',
$$
the so-called \emph{coincidence set}.

As shown in \cite{DGPT}, a successful study of the properties of the free boundary near $(x_0,t_0)\in
G(u)\cap Q_{1/2}'$ can be made by considering the rescalings
$$
u_r(x,t)=u_r^{(x_0,t_0)}(x,t)=\frac{u(x_0+rx,t_0+r^2t)}{H_u^{(x_0,t_0)}(r)^{1/2}},
$$
for $r>0$ and then studying the limits of $u_r$ as $r=r_j\to 0+$
(so-called blowups). Here
$$
H_u^{(x_0,t_0)}(r):=\frac1{r^2}\int_{t_0-r^2}^{t_0}\int_{\R^n}
u(x,t)^2\psi^2(x)\Gamma(x_0-x,t_0-t) dxdt,
$$
where $\psi(x)=\psi(|x|)$ is a cutoff function that equals $1$ on
$B_{3/4}$. Then  a point $(x_0,t_0)\in G(u)\cap B_{1/2}$ is called
regular, if $u_r$ converges in the appropriate sense to 
$$
u_0(x,t)=c_n\Re(x_{n-1}+i x_{n})^{3/2},
$$
as $r=r_j\to 0+$, after a possible rotation of coordinate axes in $\R^{n-1}$. See \cite{DGPT} for more details. 
\newcommand{\cR}{\mathcal{R}}
Thus, let $\cR(u)$
be the set of regular points of $u$. The following result has been
proved in \cite{DGPT}.

\begin{proposition}\label{prop:signor-known} Let $u$ be a solution of
  the parabolic Signorini 
  problem \eqref{eq:par-sig-1-nonhom}--\eqref{eq:par-sig-2-nonhom} in $Q_1^+$ with $g\in H^{1,1/2}(Q_1^+)$. Then the regular set
  $\cR(u)$ is a relatively open subset of $G(u)$. Moreover, if
  $(0,0)\in \cR(u)$, then there exists $\rho=\rho_u>0$ and a
  parabolically Lipschitz function $f$ such that
\begin{align*}
G(u)\cap Q_\rho'=\cR(u)\cap Q_\rho'&=G_f\cap Q_\rho'\\
\Lambda(u)\cap Q_\rho'&=E_f\cap Q_\rho'.
\end{align*}
Furthermore, for any $0<\eta<1$, we can find $\rho>0$ such that
$$
\partial_e u\geq 0\quad\text{in }Q_\rho, 
$$
for any unit direction $e\in\R^{n-1}$ such that $e\cdot e_{n-1}>\eta$
and moreover
$$
\partial_e u(x,t)\geq c\,\dist_p((x,t), E_f)\quad\text{in }Q_\rho, 
$$
for some $c>0$.\qed
\end{proposition}

We next show that an application of
Theorem~\ref{thm:nondegeneracy2-lower} implies the following result.

\begin{theorem} Let $u$ be as in Proposition~\ref{prop:signor-known}
  and $(0,0)\in\cR(u)$. Then there exist $\delta<\rho$ such that
  $\nabla''f\in H^{\alpha,\alpha/2}(Q'_{\delta})$ for some $\alpha>0$, i.e., $\cR(u)$
  has H\"older continuous spatial normals in $Q'_{\delta}$.
\end{theorem}
\begin{proof} We will work in parabolic boxes
  $\Theta_\delta=\Psi_\delta\cap \{t\leq 0\}$ instead
  of cylinders $Q_\delta$.
For a small $\epsilon>0$ let $e=(\cos\epsilon)e_{n-1}+(\sin\epsilon)e_j$ for some
$j=1,\ldots,n-2$ and consider two functions
$$
u_1=\partial_{e}u\quad\text{and}\quad u_2=\partial_{e_{n-1}}u.
$$
Then by Proposition~\ref{prop:signor-known}, the conditions of
Theorem~\ref{thm:nondegeneracy2-lower} are satisfied (after a
rescaling), provided $\cos\epsilon>\eta$. Thus, if we
fix such $\epsilon>0$, then we will have that for some $\delta>0$ and $0<\alpha<1$
$$
\frac{\partial_e u}{\partial_{e_{n-1}}u}\in
H^{\alpha,\alpha/2}(\Theta_{\delta}).               
$$
This gives that
$$
\frac{\partial_{e_j} u}{\partial_{e_{n-1}}u}\in
H^{\alpha,\alpha/2}(\Theta_{\delta}),\quad
j=1,\ldots, n-2.               
$$
Hence the level surfaces $\{u=\sigma\}\cap \Theta'_{\delta}$ are given
as graphs
$$
x_{n-1}=f_\sigma(x'',t),\quad x''\in \Theta_{\delta}'',
$$
with uniform in $\sigma>0$ estimate on
$\|\nabla''f_\sigma\|_{H^{\alpha,\alpha/2}(\Theta_{\delta}'')}$. Consequently, this
implies that 
\[
\nabla'' f\in H^{\alpha,\alpha/2}(\Theta_{\delta}''), 
\]
and completes the proof of the theorem.
\end{proof}

\end{document}